\documentclass[10pt]{amsart}
\usepackage{amssymb, amsthm, amsmath}
\usepackage[all]{xy}
\usepackage{graphicx}
\usepackage{psfrag}

\newtheorem{prop}{Proposition}
\newtheorem{thm}{Theorem}
\newtheorem{cor}{Corollary}
\newtheorem{lemma}{Lemma}
\theoremstyle{definition}

\newtheorem{defn}{Definition}
\newtheorem{example}{Example}

\newcommand\C{{\mathbb C}}

\newcommand\N{{\mathbb N}}
\newcommand{\cN}{{\mathcal N}}
\newcommand\IH{{\mathbb H}}

\newcommand{\Ti}{\Theta}

\newcommand\W{{\mathrm W}}

\newcommand\cC{{\mathcal C}}
\newcommand\cA{{\mathcal A}}

\newcommand\tm{{\mathrm{t}}}
\newcommand\Z{{\mathbb Z}}
\newcommand\cP{{\mathcal P}}

\newcommand\AS{{\mathfrak S}}

\newcommand\al{\alpha}
\newcommand{\be}{\beta}
\newcommand\la{\lambda}

\newcommand\ssm{\smallsetminus}
\newcommand\noin{\noindent}

\newcommand\bull{{\scriptscriptstyle \bullet}}
\newcommand\eqto{\stackrel{\lower1.5pt\hbox{$\scriptstyle\sim\,$}}\to}
\newcommand\ov{\overline}

\DeclareMathOperator{\Ker}{Ker}

\DeclareMathOperator{\Pf}{Pfaffian}
\DeclareMathOperator{\Sp}{Sp}

\DeclareMathOperator{\LG}{LG}
\DeclareMathOperator{\IG}{IG}

\DeclareMathOperator{\OG}{OG}

\DeclareMathOperator{\HH}{\mathrm{H}}

\newcommand{\ignore}[1]{}
\newcommand{\pic}[2]{\includegraphics[scale=#1]{#2}}

\psfrag{h}{$h$}
\psfrag{g}{$g$}
\psfrag{bh}{$b_h$}
\psfrag{gh}{$g_h$}
\psfrag{k}{$k$}

\begin{document}

\title[Double theta polynomials and equivariant Giambelli formulas]
{Double theta polynomials and equivariant Giambelli formulas}

\date{December 2, 2015}

\author{Harry~Tamvakis} \address{University of Maryland, Department of
Mathematics, 1301 Mathematics Building, College Park, MD 20742, USA}
\email{harryt@math.umd.edu}

\author{Elizabeth~Wilson} \address{University of Maryland, Department of
Mathematics, 1301 Mathematics Building, College Park, MD 20742, USA}
\email{bethmcl@math.umd.edu}


\thanks{The first author was supported in part by NSF Grant DMS-1303352.}

\begin{abstract}
We use Young's raising operators to introduce and study {\em double
  theta polynomials}, which specialize to both the theta polynomials
of Buch, Kresch, and Tamvakis, and to double (or factorial) Schur
$S$-polynomials and $Q$-polynomials. These double theta polynomials
give Giambelli formulas which represent the equivariant Schubert
classes in the torus-equivariant cohomology ring of symplectic
Grassmannians, and we employ them to obtain a new presentation of this
ring in terms of intrinsic generators and relations.
\end{abstract}

\maketitle

\setcounter{section}{-1}


\section{Introduction}

Fix the nonnegative integer $k$ and let $\IG=\IG(n-k,2n)$ be the
Grassmannian which parametrizes isotropic subspaces of dimension $n-k$
in the vector space $\C^{2n}$, equipped with a symplectic form. In
2008, Buch, Kresch, and Tamvakis \cite{BKT2} introduced the theta
polynomials $\Ti_\la(c)$, a family of polynomials in independent
variables $c_p$, defined using Young's raising operators
\cite{Y}. When the $c_p$ are mapped to the special Schubert classes,
which are the Chern classes of the tautological quotient vector bundle
over $\IG$, then the $\Ti_\la(c)$ represent the Schubert classes in
$\HH^*(\IG,\Z)$. These theta polynomials give a combinatorially
explicit, intrinsic solution to the {\em Giambelli problem} for the
cohomology ring of $\IG$.

This paper is concerned with Giambelli formulas for the {\em
  equivariant Schubert classes} in the equivariant cohomology ring
$\HH^*_T(\IG)$, where $T$ denotes a maximal torus of the complex
symplectic group.  As explained in \cite{Gr} and \cite{T3}, such
formulas are equivalent to corresponding ones within the framework of
degeneracy loci of vector bundles \cite{F}. In 2009, the theta
polynomials $\Ti_\la(c)$ were extended to obtain representing
polynomials in the torus-equivariant cohomology ring of $\IG$, and
more generally, of any isotropic partial flag variety, by the first
author \cite{T2}. The {\em equivariant Giambelli problem} was thus
solved in a uniform manner for any classical $G/P$ space, in terms of
positive combinatorial formulas which are {\em native to $G/P$}. See
\cite{T3} for an exposition of this work.

It has been known for some time (cf.\ \cite{KL,Ka, Ik, Mi2}) that the
equivariant Schubert classes of the usual (type A) Grassmannian and of
the Lagrangian Grassmannian $\LG(n,2n)$ may be represented by
Jacobi-Trudi type determinants and Schur Pfaffians which generalize
the classical results of \cite{G, P}. The polynomials which appear in
these formulas are called double (or factorial) Schur $S$-polynomials
and $Q$-polynomials, respectively (cf.\ \cite{BL, I}). Our aim here is
to use raising operators to define {\em double theta polynomials}
$\Ti_\la(c\, |\, t)$, which advance the theory of the single theta
polynomials $\Ti_\la(c)$ from \cite{BKT2} to the equivariant setting,
and specialize to both of the aforementioned versions of double Schur
polynomials. The $\Ti_\la(c\, |\, t)$ represent the equivariant
Schubert classes on $\IG$, but differ from the equivariant Giambelli
polynomials given in \cite{T2}.  It follows that the two theories must
agree up to the ideal $I^{(k)}$ of {\em classical relations}
(\ref{relations}) among the variables $c_p$. We discuss this in detail
in \S\ref{compare} and Corollary \ref{comp}.

For the rest of this paper, $c=(c_1,c_2,\ldots)$ and
$t=(t_1,t_2,\ldots)$ will denote two families of commuting
variables. We set $c_0=1$ and $c_p=0$ for any $p<0$.  For any integers
$j\geq 0$ and $r\geq 1$, define the elementary and complete symmetric
polynomials $e_j(t_1,\ldots,t_r)$ and $h_j(t_1,\ldots,t_r)$ by the
generating series
\[
\prod_{i=1}^r(1+t_iz) = \sum_{j=0}^{\infty}
e_j(t_1,\ldots,t_r)z^j \ \ \ \text{and} \ \ \ 
\prod_{i=1}^r(1-t_iz)^{-1} = \sum_{j=0}^{\infty}
h_j(t_1,\ldots,t_r)z^j,
\]
respectively. Let $e^r_j(t):=e_j(t_1,\ldots,t_r)$,
$h^r_j(t):=h_j(t_1,\ldots,t_r)$, and
$e^0_j(t)=h^0_j(t):=\delta_{0j}$, where $\delta_{0j}$ denotes the
Kronecker delta. Furthermore, if $r<0$ then define
$h^r_j(t):=e^{-r}_j(t)$.

We will work throughout with integer sequences $\al =
(\al_1,\al_2,\ldots)$ which are assumed to have finite support, when
they appear as subscripts. For any positive integers $i<j$ and integer
sequence $\alpha$, define the operator $R_{ij}$ by
$$R_{ij}(\alpha) := (\alpha_1,\ldots,\alpha_i+1,\ldots,\alpha_j-1,
\ldots).$$ A {\em raising operator} $R$ is any monomial in the basic
operators $R_{ij}$. The sequence $\al$ is a {\em composition} if
$\al_i\geq 0$ for all $i$, and a {\em partition} if $\al_i\geq
\al_{i+1}\geq 0$ for all $i\geq 1$. As is customary, we will 
identify partitions with their Young diagram of boxes.

Fix the nonnegative integer $k$. For any integers $p$ and $r$, define
\[
c^r_p:= \sum_{j=0}^p c_{p-j} \, h^r_j(-t)
\]
and for any integer sequences $\al$ and $\be$, let
\[
c^\be_{\al}:=\prod_{i \geq 1} c^{\be_i}_{\al_i} 
\ \ \ \text{and} \ \ \ 
c_\al := c_{\al}^0 = \prod_{i\geq 1} c_{\al_i}.  
\]
Given any raising operator $R$, define $R\, c^\be_{\al} :=
c^\be_{R\al}$.  It is important that the variables
$c^{\be_i}_{\al_i}$ in the monomials $c^\be_{\al}$ are regarded
as {\em noncommuting} for the purposes of this action.

We say that a partition $\la$ is {\em $k$-strict} if all parts $\la_i$
which are strictly greater than $k$ are distinct. To any such $\la$,
we attach a finite set of pairs 
\[
\cC(\la) :=\{ (i,j)\in \N\times\N \ |\ 1\leq i<j \ \ \text{and} \ \ 
\la_i+\la_j > 2k+j-i\}
\]
and a sequence $\beta(\la)=\{\beta_j(\la)\}_{j\geq 1}$ defined by
\[
\be_j(\la):=k+1-\la_j+\#\{i<j\ |\ (i,j)\notin \cC(\la)\}, \ \ \text{for
  all} \ \ j\geq 1.
\]
Following \cite{BKT2}, consider the raising operator expression
$R^\la$ given by
\[
R^{\la} := \prod_{i<j}(1-R_{ij})\prod_{(i,j)\in\cC(\la)}
(1+R_{ij})^{-1}.
\]

\begin{defn}
\label{Thdef}
For any $k$-strict partition $\la$, the {\em double theta polynomial} 
$\Ti_\la(c \, |\, t)$ is defined by 
\[
\Ti_\la(c \, |\, t) := R^\la\,c^{\be(\la)}_{\la}.
\]
The single theta polynomial $\Ti_\la(c)$ of \cite{BKT2} is given by
\[
\Ti_\la(c)=\Ti_\la(c \, |\, 0) = R^\la\, c_\la.
\] 
\end{defn}

We next relate the double theta polynomials $\Ti_\la(c\, |\, t)$ to
the torus-equivariant Schubert classes on $\IG$. Let
$\{e_1,\ldots,e_{2n}\}$ denote the standard symplectic basis of
$\C^{2n}$ and let $F_i$ be the subspace spanned by the first $i$
vectors of this basis, so that $F_{n-i}^\perp = F_{n+i}$ for $0\leq i
\leq n$. Let $B_n$ denote the stabilizer of the flag $F_\bull$ in the
symplectic group $\Sp_{2n}(\C)$, and let $T_n$ be the associated
maximal torus in the Borel subgroup $B_n$. The $T_n$-equivariant
cohomology ring $\HH^*_{T_n}(\IG(n-k,2n),\Z)$ is defined as the
cohomology ring of the Borel mixing space $ET_n\times^{T_n}\IG$. The
Schubert varieties in $\IG$ are the closures of the $B_n$-orbits, and
are indexed by the $k$-strict partitions $\la$ whose Young diagram
fits in an $(n-k)\times (n+k)$ rectangle. Any such $\la$ defines a
Schubert variety $X_\lambda=X_\la(F_\bull)$ of codimension
$|\la|:=\sum_i\la_i$ by
\begin{equation}
\label{Xlam}
   X_\lambda := \{ \Sigma \in \IG \mid \dim(\Sigma \cap
   F_{n+\beta_j(\lambda)}) \geq j \ \ \ \forall\, 1 \leq j \leq
   n-k \} \,.
\end{equation}
Since $X_\la$ is stable under the action of $T_n$, we obtain 
an {\em equivariant Schubert class} 
$[X_\la]^{T_n}:=[ET_n\times^{T_n}X_\la]$ in $\HH^*_{T_n}(\IG(n-k,2n))$.

Following \cite[\S 5.2]{BKT2} and \cite[\S 5]{IMN1}, we consider the
{\em stable} equivariant cohomology ring of $\IG(n-k,2n)$, denoted by
$\IH_T(\IG_k)$. The latter is defined by
\[
\IH_T(\IG_k) := \lim_{\longleftarrow}\HH_{T_n}^*(\IG(n-k,2n)),
\]
where the inverse limit of the system
\begin{equation}
\label{embeddings}
\cdots \rightarrow \HH^*_{T_{n+1}}(\IG(n+1-k,2n+2)) \rightarrow
\HH^*_{T_n}(\IG(n-k,2n))\rightarrow \cdots
\end{equation}
is taken in the category of graded algebras.  The surjections
(\ref{embeddings}) are induced from the natural inclusions
$W_n\hookrightarrow W_{n+1}$ of the Weyl groups of type C, as in
\cite[\S 2]{BH}, \cite[\S 10]{IMN1}, and \S \ref{trans} of the present
work. Moreover, the variables $t_i$ are identified with the characters
of the maximal tori $T_n$ in a compatible fashion, as in
loc.\ cit. One then has that $\IH_T(\IG_k)$ is a free $\Z[t]$-algebra
with a basis of stable equivariant Schubert classes
$$\sigma_\la:=\lim_{\longleftarrow}[X_\la]^{T_n},$$ one for every
$k$-strict partition $\la$. We view $\HH_{T_n}^*(\IG(n-k,2n))$ as a
$\Z[t]$-module via the natural projection map $\Z[t]\to
\Z[t_1,\ldots,t_n]$.

Consider the graded polynomial ring $\Z[c]:=\Z[c_1,c_2,\ldots]$ where
the element $c_p$ has degree $p$ for each $p\geq 1$.  Let
$I^{(k)}\subset \Z[c]$ be the homogeneous ideal generated by the
relations
\begin{equation}
\label{relations}
\frac{1-R_{12}}{1+R_{12}}\, c_{p,p} = 
c_pc_p + 2\sum_{i=1}^p(-1)^i c_{p+i}c_{p-i}= 0
\ \ \ \text{for} \  p > k,
\end{equation}
and define the quotient ring $C^{(k)}:=\Z[c]/{I^{(k)}}$.  We call
the graded polynomial ring $C^{(k)}[t]$ the {\em ring of double theta
polynomials}.

\begin{thm}
\label{mainthm} 
The polynomials $\Ti_\la(c\, |\, t)$, as $\la$ runs
over all $k$-strict partitions, form a free $\Z[t]$-basis of
$C^{(k)}[t]$. There is an isomorphism of graded $\Z[t]$-algebras
\[
\pi: C^{(k)}[t]\to \IH_T(\IG_k)
\]
such that $\Ti_\la(c\, |\, t)$ is mapped to
$\sigma_\la$, for every $k$-strict partition $\la$. For 
every $n\geq 1$, the morphism $\pi$ induces a 
surjective homomorphism of graded $\Z[t]$-algebras
$$\pi_n: C^{(k)}[t]\to \HH_{T_n}^*(\IG(n-k,2n))$$ which maps
$\Ti_\la(c\, |\, t)$ to $[X_\la]^{T_n}$, if $\la$ fits inside an
$(n-k)\times (n+k)$ rectangle, and to zero, otherwise. The equivariant 
cohomology ring $\HH_{T_n}^*(\IG(n-k,2n))$ is presented as a quotient of 
$C^{(k)}[t]$ modulo the relations 
\begin{equation}
\label{rel1}
\Ti_p(c\, |\, t)  
= 0, \ \ \ p\geq n+k+1
\end{equation}
and
\begin{equation}
\label{rel2}
\Ti_{(1^p)}(c\, |\, t) = 0, \ \ \ n-k+1 \leq p\leq n+k,
\end{equation}
where $(1^p)$ denotes the partition $(1,\ldots,1)$ with $p$ parts.
\end{thm}

Theorem \ref{mainthm} has a direct analogue for the torus-equivariant
cohomology ring of the odd orthogonal Grassmannian $\OG(n-k,2n+1)$;
this follows from known results (cf.\ \cite[\S 2.2 and \S 5.1]{T3}).
In this case, the polynomials $2^{-\ell_k(\la)}\Ti_\la(c\, |\, t)$
represent the equivariant Schubert classes, where $\ell_k(\la)$
denotes the number of parts $\la_i$ of $\la$ which are strictly bigger
than $k$. Theorem \ref{mainthm} therefore generalizes the results of
Pragacz \cite[\S 6]{P} and Ikeda and Naruse \cite{Ik, IN} to the
equivariant cohomology of all symplectic and odd orthogonal
Grassmannians.

Let
\begin{equation}
\label{ses}
0 \to E'\to E \to E''\to 0
\end{equation}
denote the universal exact sequence of vector bundles over
$\IG(n-k,2n)$, with $E$ the trivial bundle of rank $2n$ and $E'$ the
tautological subbundle of rank $n-k$. The $T_n$-equivariant vector
bundles $E'$, $E$, and $F_j$ have equivariant Chern classes, denoted
by $c_p^T(E')$, $c^T_p(E)$, and $c_p^T(F_j)$, respectively.  Define
$c^T_p(E-E'-F_j)$ by the equation of total Chern classes
$c^T(E-E'-F_j):=c^T(E)c^T(E')^{-1}c^T(F_j)^{-1}$. We now have the
following Chern class formula for the Schubert classes $[X_\la]^T$.

\begin{cor}
\label{maincor}
Let $\la$ be a $k$-strict partition that fits inside an
$(n-k)\times (n+k)$ rectangle. Then we have
\begin{equation}
\label{genKaz}
[X_\la]^T = \Ti_\la(E-E'-F_{n+\be(\la)}) = R^\la\, c^T_\la(E-E'-F_{n+\be(\la)})
\end{equation}
in the equivariant cohomology ring $\HH^*_{T_n}(\IG(n-k,2n))$. 
\end{cor}

The equivariant Chern polynomial in (\ref{genKaz}) is interpreted as the
image of $\Ti_\la(c)=R^\la\, c_\la$ under the $\Z$-linear map which
sends the noncommutative monomial $c_\al$ to $\prod_j
c^T_{\al_j}(E-E'-F_{n+\be_j(\la)})$, for every integer sequence
$\al$. Formula (\ref{genKaz}) is a natural generalization of
Kazarian's multi-Pfaffian formula \cite[Thm.\ 1.1]{Ka} to arbitrary
equivariant Schubert classes on symplectic Grassmannians.

The double theta polynomials $\Ti_\la(c\, |\, t)$ were defined in
Wilson's 2010 University of Maryland Ph.D. thesis \cite{W}, whose aim
was to apply the raising operator approach of \cite{BKT2} to the
theory of factorial Schur polynomials and equivariant Giambelli
formulas.\footnote{Wilson actually worked with the principal
specialization $\Ti_\la(x,z\, |\, t)$ of $\Ti_\la(c\, |\, t)$ in the
ring of type C double Schubert polynomials of \cite{IMN1}.} In her
dissertation, Wilson gave a direct proof that the $\Ti_\la(c\, |\, t)$
satisfy the {\em equivariant Chevalley rule} in $\IH_T(\IG_k)$, i.e.,
the combinatorial formula for the expansion of the products
$\sigma_1\cdot \sigma_\la$ in the basis of stable equivariant Schubert
classes. This fits into a program for proving Theorem \ref{mainthm},
and its extension to equivariant quantum cohomology, by applying
Mihalcea's characterization theorem \cite[Cor.\ 8.2]{Mi1} (see \S
\ref{trans}). The proof of the Chevalley rule for the $\Ti_\la(c\, |\,
t)$ which we present in \S \ref{prelim} and \S \ref{cfdtp} uses the
technical tools for working with raising operators constructed in
\cite{BKT2}, but the argument is much simpler, since it avoids the
general Substitution Rule employed in op.\ cit.

It was conjectured by Wilson in \cite{W} that the $\Ti_\la(c\, |\, t)$
represent the stable equivariant Schubert classes on $\IG$. Ikeda and
Matsumura \cite[Thm.\ 1.2]{IM} established this result, which is the
main ingredient behind the new presentation of the equivariant
cohomology ring $\HH^*_{T_n}(\IG(n-k,2n))$ displayed in Theorem
\ref{mainthm}. A key technical achievement in their proof was 
to show that the $\Ti_\la(c\, |\,t)$ are compatible with the action 
of the (left) divided difference operators on $C^{(k)}[t]$. In Section
\ref{gdds}, we explain how this fact may be combined with a formula
for the equivariant Schubert class of a point on $\IG(n-k,2n)$, to
obtain a proof of Wilson's conjecture; the argument found in \cite{IM}
is different, and uses localization in equivariant cohomology.

It is easy to see that the polynomial $\Ti_\la(c\, |\, t)$ may be
written formally as a sum of Schur Pfaffians (cf.\ Proposition
\ref{kstrictS}). This identity was used in \cite{IM} in the proof of
the compatibility of the $\Ti_\la(c\, |\,t)$ with divided differences,
allowing them to avoid the language of raising operators entirely. In
Section \ref{dds}, we eliminate this feature of their argument, and
work directly with the raising operator expressions $R^\la$. This
makes the proof more transparent, and leads naturally to a
companion theory of {\em double eta polynomials} $H_\la(c\, |\, t)$,
which represent the equivariant Schubert classes on even orthogonal
Grassmannians; see \cite{T4} for further details.

The role of the first author of this article has been mostly
expository, extending the point of view found in \cite{T2, T3} to the
present setting, and comparing the results with various earlier
formulas in the literature. We emphasize that the double theta
polynomials $\Ti_\la(c\, |\, t)$ defined here are new, and are not a
special case of the type C double Schubert polynomials of \cite{IMN1},
which represent the equivariant Schubert classes on complete
symplectic flag varieties. Recently, Hudson, Ikeda, Matsumura, and
Naruse \cite{HIMN} proved an analogue of our cohomological formula in
connective $K$-theory, and Anderson and Fulton \cite{AF} obtained
related Chern class formulas for more general degeneracy loci. The
latter work features a powerful geometric approach to the theory,
which is uniform for all four classical Lie types.

This paper is organized as follows. Section \ref{prelim} establishes
some preliminary lemmas which are required in our proof that the
$\Ti_\la(c\, |\, t)$ obey the equivariant Chevalley rule (Theorem
\ref{cvthm}); the latter is completed in Sections \ref{cfdtp} and
\ref{trans}.  Although Theorem \ref{cvthm} is a consequence of Theorem
\ref{mainthm}, it is a much earlier result, and the independent proof
given here exhibits alternating properties of the $\Ti_\la(c\, |\, t)$
which are useful in other contexts. Section \ref{firstexs} is
concerned with how the $\Ti_\la(c\, |\, t)$ behave in special cases,
and relates them to some earlier formulas. In Section \ref{dds}, we
observe that the action of the divided differences on $C^{(k)}[t]$
lifts to $\Z[c,t]$, and prove the key result about them and the
$\Ti_\la(c\, |\, t)$ (Proposition \ref{uniq}) using our raising
operator approach.  The proof of Theorem \ref{mainthm} and its
corollaries is contained in Section \ref{gdds}.

We thank Andrew Kresch and Leonardo Mihalcea for helpful discussions
related to this project, and Takeshi Ikeda and Tomoo Matsumura for
informative correspondence regarding their joint work.

\section{Preliminary results}
\label{prelim}

We begin with a purely formal lemma regarding the elements
$c_p^r$ of $\Z[c,t]$. For $r<0$ we let $t_r:=t_{-r}$.

\begin{lemma}
\label{ctlem}
Suppose that $p,r\in \Z$.

\medskip
\noin
{\em (a)} Assume that $r>0$. Then we have
\[
c_p^r = c_p^{r-1} - t_r\, c_{p-1}^r.
\]

\medskip
\noin
{\em (b)} Assume that $r \leq 0$. Then we have
\[
c_p^r = c_p^{r-1} + t_{r-1}\, c_{p-1}^r.
\]
\end{lemma}
\begin{proof}
For any positive integer $m$, we have the basic equations
\[
e^m_j(-t) = e^{m-1}_j(-t)-t_m\,e_{j-1}^{m-1}(-t) \ \  \text{and}  \ \ 
h^m_j(-t) = h^{m-1}_j(-t)-t_m\,h_{j-1}^m(-t).
\]
If $r > 0$, it follows that
\[
c_p^r = \sum_{j=0}^p
c_{p-j}\, h_j^r(-t) = \sum_{j=0}^p c_{p-j}
(h_j^{r-1}(-t)-t_r h_{j-1}^r(-t))
\]
and the result of part (a) is proved. For (b), let $m:=|r|+1$ and compute
\begin{gather*} 
c_p^r = \sum_{j=0}^{m-1} c_{p-j}\, e_j^{m-1}(-t)
= \sum_{j=0}^{m-1} c_{p-j} 
(e_j^m(-t)+t_m\,e_{j-1}^{m-1}(-t)) \\
=\sum_{j=0}^{m-1} c_{p-j}\, e_j^m(-t) +
t_m\sum_{j=1}^{m-1} c_{p-j}\,e_{j-1}^{m-1}(-t) \\
=\sum_{j=0}^m c_{p-j}\, e_j^m(-t) +
t_m\sum_{s=0}^{m-1} c_{p-1-s}\,e_s^{m-1}(-t).
\end{gather*}
(In the last equality, notice that we have added the term 
$$A:=c_{p-m}\,(-t_1)\cdots (-t_m)$$ 
to the first sum, and added $-A$ to the second sum.) The result follows.
\end{proof}

Let $\Delta^{\circ} := \{(i,j) \in \N \times \N \mid 1\leq i<j \}$,
equipped with the partial order $\leq$ defined by $(i',j')\leq (i,j)$
if and only if $i'\leq i$ and $j'\leq j$.  A finite subset $D$ of
$\Delta^{\circ}$ is a {\em valid set of pairs} if it is an order ideal
in $\Delta^{\circ}$. An {\em outer corner} of a valid set of pairs $D$
is a pair $(i,j)\in \Delta^\circ\smallsetminus D$ such that $D\cup
\{(i,j)\}$ is also a valid set of pairs.

\begin{defn}
Given a valid set of pairs
$D$ and an integer sequence $\mu$, we 
define sequences $a=a(D)$ and $\gamma=\gamma(D,\mu)$ by
\[
a_j(D) := \#\{i<j\ |\ (i,j)\notin D\}  \ \ \text{ and } \ \  
\gamma_j(D,\mu) :=k+1-\mu_j+a_j(D)
\]
and the raising operator
\[
R^D := \prod_{i<j}(1-R_{ij})\prod_{i<j\, :\, (i,j)\in D}(1+R_{ij})^{-1}.
\]
Finally, let
$$T(D,\mu):=R^D\, c^{\gamma(D,\mu)}_{\mu}.$$ 
\end{defn}

Let $\epsilon_j$ denote the 
$j$-th standard basis vector in $\Z^\ell$.
The next result is the analogue of \cite[Lemma 1.4]{BKT2} that 
we require here.

\begin{lemma}
\label{tident}
Let $(i,j)$ be an outer corner of the valid set of pairs $D$, and
suppose that the integer sequence $\mu$ satisfies $\mu_i > k \geq
\mu_j$.  Let $\rho:=r\epsilon_j$ for some integer $r\leq 1$ and 
set $\gamma := \gamma(D,\mu + \rho)$.  Then we have
\begin{align*}
T(D,\mu+\rho) & = 
T(D\cup (i,j),\mu+ \rho)+T(D\cup (i,j), \mu+ R_{ij}\rho) \\
&\quad  + (t_{\gamma_i-1}-t_{\gamma_j})\,
R^{D\cup(i,j)}\, c^{\gamma}_{\mu+\rho-\epsilon_j}
\end{align*}
in $\Z[c,t]$. In particular, if $\mu_i+\mu_j+\rho_j=2k+1+a_j(D)$, then 
\begin{equation}
\label{cancel}
T(D,\mu+\rho) = T(D\cup (i,j),\mu+\rho)+
T(D\cup (i,j),\mu+ R_{ij}\rho).
\end{equation}
\end{lemma}
\begin{proof}
Observe that since $(i,j)$ is an outer corner of $D$, we have 
$$a_i(D)=a_i(D\cup(i,j))=0,$$ while $a_j(D\cup(i,j))=a_j(D)-1$. It
follows that 
$$\gamma(D\cup (i,j),\mu+\rho) = \gamma(D,\mu+\rho) - \epsilon_j.$$
For any integer $p$, Lemma \ref{ctlem}(a) gives
\[
c_p^{\gamma_j} = c_p^{\gamma_j-1} - t_{\gamma_j}
\, c_{p-1}^{\gamma_j}
\]
and hence
\begin{equation}
\label{qu}
R^{D\cup(i,j)}\, c^{\gamma}_{\mu+\rho} = T(D\cup(i,j),\mu+\rho)
-t_{\gamma_j}R^{D\cup(i,j)}\, c^{\gamma}_{\mu+\rho-\epsilon_j} .
\end{equation}
Notice that $\gamma_i=k+1-\mu_i\leq 0$. Therefore Lemma \ref{ctlem}(b) gives
\[
c_{p+1}^{\gamma_i} = c_{p+1}^{\gamma_i-1} + t_{\gamma_i-1}\, c_p^{\gamma_i}
\]
and we deduce that 
\begin{equation}
\label{CCequ}
R^{D\cup(i,j)}\, c^{\gamma}_{\mu+R_{ij}\rho} = T(D\cup(i,j),\mu+R_{ij}\rho)
+ t_{\gamma_i-1}\, R^{D\cup(i,j)}\, c^{\gamma}_{\mu+\rho-\epsilon_j} .
\end{equation}
Since we have
\[
T(D,\mu+\rho) = R^D\, c^{\gamma}_{\mu+\rho} =
R^{D\cup(i,j)}\, c^{\gamma}_{\mu+\rho} + R^{D\cup (i,j)} \, 
c^{\gamma}_{\mu+R_{ij}\rho},
\]
the first equality follows by combining (\ref{qu}) with (\ref{CCequ}).
To prove (\ref{cancel}), note that $|\gamma_i-1|=\mu_i-k$ and
$\gamma_j=k+1+a_j(D)-\mu_j-\rho_j$, therefore
$t_{\gamma_i-1}-t_{\gamma_j}=0$.
\end{proof}

\begin{lemma}\label{commuteA}
Let $\la=(\la_1,\ldots,\la_{j-1})$ and $\mu=(\mu_{j+2},\ldots,\mu_\ell)$
be integer vectors. Let $D$ be a valid set of
pairs such that $(j,j+1)\notin D$ and for each $h<j$, $(h,j)\in D$ if
and only if $(h,j+1)\in D$.  Then for any integers $r$ and $s$, we have
\[ 
T(D,(\la,r,s,\mu))=-T(D,(\la,s-1,r+1,\mu))
\]
in $\Z[c,t]$. In particular, if $\la_{j+1}=\la_j+1$, then $T(D,\la)=0$.
\end{lemma}
\begin{proof}
If $\al=(\al_1,\ldots,\al_\ell)$ is any integer sequence and $1\leq j
< \ell$, define $$s_j(\al):=(\al_1,\ldots, \al_{j-1},
\al_{j+1}, \al_j,\al_{j+2},\dots, \al_\ell).$$ 
The hypotheses on $D$ and $j$ imply that 
\[
\gamma(D, (\la,s-1,r+1,\mu))=
s_j(\gamma(D,(\la,r,s,\mu))).
\]
The argument is now the same as in the proof of \cite[Lemma 1.2]{BKT2};
the details are spelled out in \cite[Lemma 10]{W}.
\end{proof}

Throughout this paper, we will often write equalities that hold only 
in the rings $C^{(k)}$ and $C^{(k)}[t]$, where we have imposed the 
relations (\ref{relations}) on the generators $c_p$. Whenever these
relations are needed, we will emphasize this by noting that the 
equalities are true in $C^{(k)}[t]$ rather than in $\Z[c,t]$.

\begin{lemma}
\label{R12lem}
Given $p>k$ and $q>k$, we have
\[
\frac{1-R_{12}}{1+R_{12}}\,c_{(p,q)}^{(k+1-p,k+1-q)}=
-\frac{1-R_{12}}{1+R_{12}}\,c_{(q,p)}^{(k+1-q,k+1-p)}
\]
in $C^{(k)}[t]$.
\end{lemma}
\begin{proof}
The relations (\ref{relations}) readily imply that 
\[
\frac{1-R_{12}}{1+R_{12}}\, c_{(r,s)} = -
\frac{1-R_{12}}{1+R_{12}}\, c_{(s,r)}
\]
whenever $r+s>2k$ (see \cite[Eqn.\ (9)]{BKT2}). 
For any raising operator $R$, we have
\[
R\,c_{(p,q)}^{(k+1-p,k+1-q)} = \sum_{i=0}^p\sum_{j=0}^q
R\,c_{(p-i,q-j)}\,e^{p-k-1}_i(-t)e^{q-k-1}_j(-t).
\]
We deduce that
\begin{align*}
\frac{1-R_{12}}{1+R_{12}}\,c_{(p,q)}^{(k+1-p,k+1-q)} &=
 \sum_{i=0}^p\sum_{j=0}^q
\frac{1-R_{12}}{1+R_{12}} \, c_{(p-i,q-j)}\,e^{p-k-1}_i(-t)e^{q-k-1}_j(-t) \\ 
&= - \sum_{j=0}^q\sum_{i=0}^p
\frac{1-R_{12}}{1+R_{12}} \,c_{(q-j,p-i)}\,e^{q-k-1}_j(-t)e^{p-k-1}_i(-t) \\
&= -\frac{1-R_{12}}{1+R_{12}}\,c_{(q,p)}^{(k+1-q,k+1-p)}.
\end{align*}
\end{proof}

Lemma \ref{R12lem} admits the following generalization.

\begin{lemma}\label{commuteC}
Let $\la=(\la_1,\ldots,\la_{j-1})$ and $\mu=(\mu_{j+2},\ldots,\mu_\ell)$
be integer vectors. Let $D$ be a valid set of pairs
such that $(j,j+1)\in D$ and for each $h>j+1$, $(j,h)\in D$ if and
only if $(j+1,h)\in D$. If $r>k$ and $s>k$, then we have
\[ 
T(D,(\la,r,s,\mu))=-T(D,(\la,s,r,\mu))
\]
in $C^{(k)}[t]$. In particular, if $\la_{j+1}=\la_j$, then $T(D,\la)=0$.
\end{lemma}
\begin{proof}
The hypotheses on $D$ and $j$ imply that 
\[
\gamma(D,(\la,s,r,\mu))=
s_j(\gamma(D,(\la,r,s,\mu))).
\]
The rest of the proof follows
\cite[Lemma 1.3]{BKT2}, using Lemma \ref{R12lem} in place of 
\cite[Eqn.\ (9)]{BKT2}; see \cite[Lemma 11]{W} for the details.
\end{proof}

\section{The Chevalley formula for double theta polynomials}
\label{cfdtp}

We begin this section with a basis theorem for the $\Z[t]$-algebra 
$C^{(k)}[t]$.

\begin{prop}
\label{basisthm}
The monomials $c_\la$, the single theta polynomials $\Ti_\la(c)$, and 
the double theta polynomials $\Ti_\la(c\, |\, t)$ form three 
$\Z[t]$-bases of $C^{(k)}[t]$, as $\la$ runs over all 
$k$-strict partitions.
\end{prop}
\begin{proof}
It follows from \cite[Prop.\ 5.2]{BKT2} that the monomials 
$c_\la$ and the single theta polynomials $\Ti_\la(c)$ 
for $\la$ a $k$-strict partition form two $\Z$-bases of 
$C^{(k)}$. We deduce that these two families are also $\Z[t]$-bases
of $C^{(k)}[t]$. By expanding the raising operator definition of 
$\Ti_\la(c\, |\, t)$, we obtain that
\[
\Ti_\la(c\, |\, t) = c_\la + \sum_{\mu} a_{\la\mu}\, c_\mu
\]
where $a_{\la\mu}\in \Z[t]$ and the sum is over $k$-strict partitions
$\mu$ with $\mu \succ \la$ in dominance order or $|\mu|<|\la|$.
Therefore, the $\Ti_\la(c\, |\,t)$ for $\la$ $k$-strict form another
$\Z[t]$-basis of $C^{(k)}[t]$.
\end{proof}

For the remainder of this section, $\Ti_\la$ will be used to denote
$\Ti_\la(c\, |\, t)$.  To state the Chevalley formula for the double
theta polynomials $\Ti_\la$, we need to recall certain definitions
from \cite[\S 1.2]{BKT1} for the Pieri products with the divisor
class.

We say that the box $[r,c]$ in row $r$ and column $c$ of a Young
diagram is {\em $k$-related} to the box $[r',c']$ if $|c-k-1|+r =
|c'-k-1|+r'$. The two grey boxes in the Young diagram of Figure
\ref{krelated} are $k$-related.
\begin{figure}
\centering
\pic{0.65}{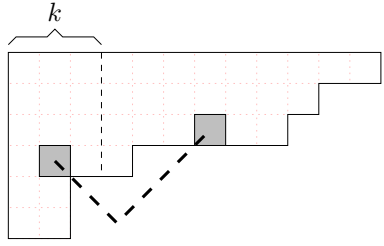} 
\caption{Two $k$-related boxes in a $k$-strict Young diagram.}
\label{krelated}
\end{figure}

For any two $k$-strict partitions $\lambda$ and $\mu$,
we write $\lambda \to \mu$ if $\mu$ may be obtained by (i) adding a
box to $\la$ or (ii) removing $r$ boxes from one of the first $k$
columns of $\la$ and adding $r+1$ boxes to a single row of the result,
so that the removed boxes and the bottom box of $\mu$ in that column
are each $k$-related to one of the added boxes. If $\la \to \mu$, let
$e_{\la\mu}=2$ if $\mu\supset \la$ and the added box in $\mu$ is not
in column $k+1$ and is not $k$-related to a bottom box in one of the
first $k$ columns of $\la$, and otherwise set $e_{\la\mu}=1$.

Fix a $k$-strict partition $\la$ and let $\ell=\ell(\la)$ denote the
{\em length} of $\la$, that is, the number of non-zero parts $\la_i$.
Recall that the {\em $k$-length} of $\la$ is 
$$\ell_k=\ell_k(\la):=\#\{\la_i\ |\ \la_i >k\}.$$

\begin{thm}[\cite{W}]
\label{cvthm} 
The equation
\begin{equation}
\label{cr}
\Ti_1\cdot\Ti_\la = \left(\sum_{j=k+1}^{k+\ell} t_j +
\sum_{j=1}^{\ell_k} t_{\la_j-k} -\sum_{j=\ell_k+1}^{\ell} t_{\beta_j(\la)}\right)
 \Ti_\la + \sum_{\la \to \mu} e_{\lambda\mu} \, \Ti_\mu
\end{equation}
holds in $C^{(k)}[t]$, where the sum is over all $k$-strict partitions
$\mu$ with $\la\to\mu$.
\end{thm}
\begin{proof}
We begin by studying the partitions $\mu$ which appear on the 
right hand side of the product rule (\ref{cr}). It is an easy
exercise to show that $\cC(\mu)\supset \cC(\la)$ for all such $\mu$
(a more general result is proved in \cite[Lemma 2.1]{BKT2}). 

\begin{defn}
Suppose that $\mu$ is a $k$-strict partition such that $\la \to
\mu$. We say that $\mu$ and the Chevalley term $\Ti_\mu$ are {\em
of type I} if $\cC(\mu)=\cC(\la)$; otherwise, $\mu$ and $\Ti_\mu$ are
said to be {\em of type II}.
\end{defn} 

Suppose that $\la \to \mu$. It is then straightforward to check the
following assertions: First, if $\mu$ is of type I, then
$e_{\la\mu}=1$.  Second, if $\mu$ is of type II, then exactly one of
the following three possibilities holds: (a) $e_{\la\mu}=2$, in which
case $\mu\ssm\la$ is a single box in some row $h$, and
$\cC(\mu)\ssm\cC(\la)$ consists of a single pair in row $h$; (b)
$\mu\ssm\la$ is a single box in some row $h$ with $\mu_h\leq k$, and the
pairs in $\cC(\mu)\ssm\cC(\la)$ are all contained in column $h$; or
(c) $\mu\ssm\la$ has $r+1$ boxes in some row $h$ and $\cC(\mu)\ssm\cC(\la)$
consists of $r$ pairs in row $h$. We say that $\mu$ and the Chevalley
term $\Ti_\mu$ are of type I, II(a), II(b), or II(c), respectively.

Define $\cC:=\{(i,j)\in \cC(\la)\ |\ j\leq \ell\}$. We represent the
valid set of pairs $\cC$ as positions above the main diagonal of a 
matrix, denoted by dots in Figure \ref{bgh}. For $1 \leq h \leq
\ell_k+1$, we let $b_h := \min \{j > \ell_k \mid (h,j) \not\in \cC \}$
and $g_h := b_{h-1}$ (by convention, we set $g_1 = \ell+1$). Figure
\ref{bgh} illustrates these invariants.

\begin{figure}\centering
\includegraphics[scale=0.6]{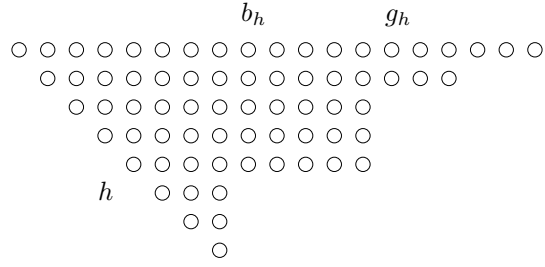}
\caption{A set of pairs $\cC$ with $h=6$, $b_h=10$, and $g_h=15$.}
\label{bgh}
\end{figure}

\begin{lemma}
\label{lambdabg}
  If $2\leq h\leq \ell_k+1$ then we have $\la_{h-1}-\la_h \geq g_h-b_h+1$.
\end{lemma}
\begin{proof}
The result is the same as \cite[Lemma 3.5]{BKT2}, but we repeat the
argument here for completeness. The inequality is clear if $b_h=g_h$,
as $\la$ is $k$-strict and $h\leq \ell_k+1$. If $b_h<g_h$, then since
$(h-1,g_h-1)\in \cC$ and $(h,b_h)\notin \cC$, we obtain
\[
\la_{h-1}-\la_h > (2k+g_h-h-\la_{g_h-1}) + (\la_{b_h}+h-b_h-2k)\geq g_h-b_h.
\]
\end{proof}

For any $h\geq 1$, let $\la^h:=\la+\epsilon_h$. Set $\gamma' :=
(\beta_1(\la),\ldots,\beta_\ell(\la),0)$ and
$$\gamma := \gamma(\cC,\la^{\ell+1})=
(\beta_1(\la),\ldots,\beta_\ell(\la),k+\ell).$$
For any $d\geq 1$ define the raising operator $R^{\la}_d$ by
\[
R_d^{\la} := \prod_{1\leq i<j\leq d}(1-R_{ij})\,
\prod_{i<j\, :\, (i,j)\in\cC} (1+R_{ij})^{-1}.
\]
We compute that 
\begin{gather*}
c_1\cdot \Ti_\la = c_1\cdot R^\la_\ell\, c^{\beta(\la)}_{\la} =
R_{\ell+1}^\la\cdot \prod_{i=1}^\ell(1-R_{i,\ell+1})^{-1} \,
c^{\gamma'}_{\la^{\ell+1}} \\ =
R^\la_{\ell+1}\cdot\prod_{i=1}^\ell(1+R_{i,\ell+1})\,c^{\gamma'}_{\la^{\ell+1}}
= R^\la_{\ell+1}\cdot (1+ \sum_{i=1}^\ell R_{i,\ell+1})
\,c^{\gamma'}_{\la^{\ell+1}}
\end{gather*}
and therefore
\begin{equation}\label{initequiv}
c_1\cdot\Ti_\la = \sum_{h=1}^{\ell+1} R^{\cC}\,
c^{\gamma'}_{\la^h}.
\end{equation}
We wish to replace $\gamma'$ by $\gamma$ in (\ref{initequiv}); note
that this substitution only affects the term $R^{\cC}\,
c^{\gamma'}_{\la^{\ell+1}}$.  
Since
$c^0_1 = c_1  = c^{k+\ell}_1 + \sum_{j=1}^{k+\ell} t_j$, we have
\[
R^{\cC}\, c^{\gamma'}_{\la^{\ell+1}} = R^{\cC}\,
c^{\gamma}_{\la^{\ell+1}} + \left(\sum_{j=1}^{k+\ell}
t_j\right) \Ti_\la
\]
and hence 
\begin{equation}\label{initequiv2}
c_1\cdot\Ti_\la  =  \left(\sum_{j=1}^{k+\ell} t_j\right)
\Ti_\la + \sum_{h=1}^{\ell+1} R^{\cC}\, c^{\gamma}_{\la^h}.
\end{equation}

We examine the terms $R^{\cC}\, c^{\gamma}_{\la^h}$ for $h\in
[1,\ell+1]$, which appear on the right hand side of equation
(\ref{initequiv2}). If $h\leq \ell_k$, then $\la_h>k$ and
$\gamma_h=k+1-\la_h\leq 0$. For any integer $p$, Lemma \ref{ctlem}(b)
gives
\[
c_{p+1}^{\gamma_h} = t_{\la_h-k} \, c_p^{k+1-\la_h} +
c_{p+1}^{k-\la_h}.
\]
It follows that
\[
R^{\cC}\, c^{\gamma}_{\la^h} = t_{\la_h-k}\,\Ti_\la +
T(\cC,\la^h).
\]
If $\ell_k < h \leq \ell$, then $\la_h\leq k$ and 
$\gamma_h=k+1+a_h-\la_h>0$. Lemma \ref{ctlem}(a) gives
\[
c_{p+1}^{\gamma_h} = - t_{\gamma_h} c_p^{\gamma_h} +
c_{p+1}^{\gamma_h-1}
\]
and hence
\[
R^{\cC}\, c^{\gamma}_{\la^h} = -t_{\gamma_h}\Ti_\la +
T(\cC,\la^h).
\]
Finally, we have $R^{\cC}\, c^{\gamma}_{\la^{\ell+1}} =
T(\cC,\la^{\ell+1})$.  We deduce that
\begin{equation}
\label{basic}
c_1\cdot\Ti_\la = \left(\sum_{j=1}^{k+\ell} t_j +
\sum_{j=1}^{\ell_k} t_{\la_j-k} -\sum_{j=\ell_k+1}^{\ell} t_{\beta_j(\la)}\right)
 \Ti_\la + \sum_{h=1}^{\ell+1} T(\cC,\la^h).
\end{equation}

We claim that
\begin{equation}
\label{subs}
\sum_{h=1}^{\ell+1} T(\cC,\la^h)
= \sum_{\la \to \mu} e_{\lambda\mu} \, \Ti_\mu.
\end{equation}
The proof of (\ref{subs}) proceeds by showing that for every $h\in
[1,\ell+1]$, the term $T(\cC,\la^h)$ is equal to (i) zero, or
(ii) a Chevalley term $\Ti_\mu$, or (iii) a sum $\Ti_\mu+\Ti_{\mu'}$
of two distinct Chevalley terms, exactly one of which is of type
II(a). The terms $T(\cC,\la^h)$ are decomposed exactly as in
the non-equivariant setting of \cite{BKT2}, using Lemmas \ref{tident},
\ref{commuteA}, and \ref{commuteC} as the main tools. However the
analysis here is much easier, so we provide a short self-contained
argument below.

Given any integer sequence $\mu$, define a weight condition $\W(i,j)$
on $\mu$ by
\[
\W(i,j)\ :\ \mu_i+\mu_j > 2k + j-i \,.
\]

\medskip
\noindent
{\bf Case 1.} Terms $T(\cC,\la^h)$ with $(h-1,h)\notin \cC$
(this condition always holds if $h>\ell_k+1$). Notice that we 
must have $\la_h \leq k$.

\medskip
\noin {\bf a)} Suppose that there is no outer corner $(r,h)$ of $\cC$
in column $h$ such that $\W(r,h)$ holds for $\la^h$. If
$\la_h=\la_{h-1}$, we have $T(\cC,\la^h)=0$ by Lemma \ref{commuteA}
with $j=h-1$ (note that $(r,h-1)\in \cC$ implies that $(r,h)\in\cC$,
so the conditions of the lemma hold). Otherwise,
$T(\cC,\la^h)=\Ti_{\la^h}$ is a Chevalley term of type I (in the
special case that $h=\ell_k+1$ we must have $\la_h < k$, so the result
is still true).

\medskip
\noin
{\bf b)} If there exists an outer corner $(r,h)$ of $\cC$ in column
$h$ such that $\W(r,h)$ holds for $\la^h$, then let $s<h$ be maximal
such that $\W(s,h)$ holds for $\la^h$. Observe that in this case, 
\[
\la_s +\la_{h-1} \geq \la_s+\la_h \geq 2k+h-s > 2k + (h-1)-s,
\]
and hence $(s,h-1)\in \cC$ or $s=h-1$. Notice also that we must have
$\la_{i-1}=\la_i+1$ for $i=r+1,\ldots,s$. 

By definition, we have $\gamma_h=k+1+a_h-\la_h$ and $\gamma_r=k+1-\la_r$.
We claim that $\la_r+\la_h=2k+a_h$. Indeed,
we know that $(r,h)\notin\cC$ and $(\la_h+1)+\la_r > 2k+h-r$, therefore
$\la_r+\la_h=2k+h-r$. The claim now follows, since $a_h = h-r$.

Let 
\[
E_j:=\{(r,h),\ldots,(j,h)\}, \ \ \ \text{for } j \in [r,s].
\]
Using Lemma \ref{tident} and the claim gives
\[
T(\cC,\la^h) = T(\cC\cup E_r,\la^h)+T(\cC\cup E_r,\la^r).
\]
Continuing in the same manner, we obtain
\[
T(\cC,\la^h) = T(\cC\cup E_s, \la^h) + 
\sum_{j=r}^s T(\cC\cup E_j,\la^j).
\]
Since $\la_{i-1}=\la_i+1$ for $i=r+1,\ldots,s$, Lemma \ref{commuteC}
implies that all terms in the sum over $j$ vanish, except for the
initial one. Therefore, we have
\[
T(\cC,\la^h) = T(\cC\cup E_r,\la^r) + T(\cC\cup E_s, \la^h).
\]
If $\la_h=\la_{h-1}$, then Lemma \ref{commuteA} implies that
$T(\cC\cup E_s, \la^h) = 0$. This must be the case since $\la_h+1 >
\la_{h-1}$ implies that $(s+1,h)$ is not an outer corner of $\cC\cup
E_s$. If $h=\ell_k+1$ and $\la_h=k$, then $\la_{h-1}=k+1$, $s=h-1$,
and $T(\cC\cup E_s,\la^h)= T(\cC\cup E_{h-1},\la^h)=0$, by
Lemma \ref{commuteC}.  The terms that remain after these cancellations
are Chevalley terms; note that $T(\cC\cup E_r,\la^r)=\Ti_{\la^r}$ is
of type II(a), and $T(\cC\cup E_s, \la^h)=\Ti_{\la^h}$ is of
type II(b).

\medskip
\noindent
{\bf Case 2.} Terms $T(\cC,\la^h)$ with $(h-1,h)\in \cC$ (and
hence $h\leq \ell_k+1$). Notice that $a_h=0$ for all such $h$, so
$\gamma_h=k+1-\la_h$.

\medskip
\noin {\bf a)} Suppose that there is no outer corner $(h,c)$ of $\cC$
such that $\W(h,c)$ holds for $\la^h$.  If $\la_h>k$ and
$\la_{h-1}=\la_h+1$, then we have $T(\cC,\la^h)=0$ by Lemma
\ref{commuteC}. In all other cases, $T(\cC,\la^h)=\Ti_{\la^h}$ is a
Chevalley term of type I.

\medskip
\noin
{\bf b)} Suppose that there exists an outer corner $(h,c)$ of $\cC$
such that $\W(h,c)$ holds for $\la^h$.  It follows that
$b_h=c$ and $\la_h+\la_c = 2k+c-h$.  Setting
$g=g_h$, we have $(h-1,g)\notin\cC$, and hence
$\la_{h-1}+\la_g \leq 2k+g-h+1$.
Lemma \ref{lambdabg} implies that 
\begin{gather*}
g-c+1 \leq \la_{h-1}-\la_h = \la_{h-1} + \la_c -2k-c+h \\
\leq (2k+g-h+1-\la_g) + \la_c -2k-c+h = g-c+1+(\la_c-\la_g). 
\end{gather*}
We deduce that if $\la_c=\la_g$, then $\la_{h-1}-\la_h=g-c+1$.

Note that we have $\la_h>k$. Let
$d\in [c,g)$ be maximal such that $\la_c=\la_d$. Define
\[
E'_j:=\{(h,c),\ldots,(h,j)\} \ \ \text{and} \ \ R_j := R_{hc}\cdots R_{hj}.
\]
Applying Lemma \ref{tident} repeatedly, and using
$\la_h+\la_c = 2k+c-h$, we obtain
\begin{equation}
\label{last}
T(\cC,\la^h) = T(\cC\cup E'_c, \la^h) + T(\cC\cup E'_d, R_d\la^h) + 
\sum_{j=c+1}^d T(\cC\cup E'_j,R_{j-1}\la^h).
\end{equation}

Lemma \ref{commuteA} implies that all the terms in the sum in (\ref{last}) 
are zero. Therefore
\[
T(\cC,\la^h) = T(\cC\cup E'_c, \la^h) + T(\cC\cup E'_d, R_d\la^h).
\]
If $\la_c>\la_g$, then we obtain two Chevalley terms; the first
summand $\Ti_{\la^h}$ is of type II(a), and the second
$\Ti_{R_d\la^h}$ is of type II(c), with $r=d-c+1$.  If $\la_c=\la_g$,
then since $\la_{h-1}-\la_h=g-c+1$, we conclude from Lemma
\ref{commuteC} that the last term $T(\cC\cup E'_d, R_d\la^h)$
vanishes.

\medskip
\noin The claim (\ref{subs}) now follows easily, by combining the
above two cases.  Observe in particular that each of the Chevalley
terms $\Ti_\mu$ of type II(a) appears twice in the decomposition of
the left hand side of (\ref{subs}): once from a term $T(\cC,\la^h)$
with $\mu\neq \la^h$, in Case 1(b), and once from a term
$T(\cC,\la^h)$ with $\mu=\la^h$, in Case 2(b). Finally, equations
(\ref{basic}) and (\ref{subs}) imply (\ref{cr}), since
$\Ti_1=c_1^k=c_1-(t_1+\cdots +t_k)$.
\end{proof}

\begin{example}
Let $k=2$ and consider the $2$-strict partition $\la=(7,4,3,2,1,1)$.
The element of the Weyl group $W_8$ associated to $\la$ is
$w_\la=(4,8,\ov{5},\ov{2},\ov{1},3, 6, 7)$ ($w_\la$ is defined in \S
\ref{trans}). We have $\cC=\{(1,2), (1,3), (1,4), (2,3)\}$,
$\gamma=\gamma(\cC,\la^7)=(-4,-1,0,3,6,7,8)$, $\ell_k=3$, and
$\ell=6$.  The Chevalley formula in this case is
\begin{align*}
\Ti_1\cdot \Ti_\la &= (t_1+t_2+t_4+2t_5+t_8)\,\Ti_\la +
\Ti_{(7,4,3,2,1,1,1)}+\Ti_{(7,4,3,2,2,1)} \\
&\quad + 2\,\Ti_{(7,5,3,2,1,1)}+ \Ti_{(7,6,3,1,1,1)}
+2\,\Ti_{(8,4,3,2,1,1)} + \Ti_{(10,4,3,2)}. 
\end{align*}
In the above rule, the partition $(7,4,3,2,1,1,1)$ is of type I, 
$(7,5,3,2,1,1)$ and $(8,4,3,2,1,1)$ are of type II(a), 
$(7,4,3,2,2,1)$ is of type II(b), and $(7,6,3,1,1,1)$ and 
$(10,4,3,2)$ are of type II(c). The terms $T(\cC,\la^h)$ in 
the sum (\ref{subs}) are expanded as follows:
\begin{gather*}
T(\cC,\la^1)=T(\cC\cup E'_5,\la^1)+ T(\cC\cup
  E'_6,R_6\la^1)=\Ti_{(8,4,3,2,1,1)}+\Ti_{(10,4,3,2)}\,;
\\ T(\cC,\la^2)=T(\cC\cup E'_4,\la^2)+ T(\cC\cup
  E'_4,R_4\la^2)=\Ti_{(7,5,3,2,1,1)}+\Ti_{(7,6,3,1,1,1)}\,;
\\ T(\cC,\la^3)=T(\cC,\la^6)=0\,; \ \ \quad
   T(\cC,\la^7)=\Ti_{(7,4,3,2,1,1,1)}\,;
   \\ T(\cC,\la^4)= T(\cC\cup E_2,\la^2) =
   \Ti_{(7,5,3,2,1,1)}\,; \\ T(\cC,\la^5)=T(\cC\cup
     E_1,\la^5)+ T(\cC\cup E_1,\la^1)=
   \Ti_{(7,4,3,2,2,1)}+\Ti_{(8,4,3,2,1,1)}\,.
\end{gather*}
\end{example}

\section{Transition to the hyperoctahedral group}
\label{trans}

The Weyl group for the root system of type $\text{C}_n$ is the {\em
  hyperoctahedral group} $W_n$ of signed permutations on the set
$\{1,\ldots,n\}$. We adopt the notation where a bar over an integer
denotes a negative sign.  The group $W_n$ is generated by the simple
transpositions $s_i=(i,i+1)$ for $1\leq i \leq n-1$ and the sign
change $s_0(1)=\ov{1}$.  The natural inclusion $W_n\hookrightarrow
W_{n+1}$ is defined by adding the fixed point $n+1$, and we set
$W_\infty=\cup_n W_n$. The length of an element $w$ in $W_\infty$ is
denoted by $\ell(w)$.

An element $w\in W_\infty$ is called {\em $k$-Grassmannian} if we have
$\ell(ws_i)=\ell(w)+1$ for all $i\neq k$; the set of all
$k$-Grassmannian elements is denoted by $W^{(k)}$. The elements of
$W^{(k)}$ are the minimal length coset representatives in
$W_\infty/W_{(k)}$, where $W_{(k)}$ is the subgroup of $W_\infty$
generated by the $s_i$ for $i\neq k$.  There is a canonical bijection
between $k$-strict partitions and $W^{(k)}$ (see \cite[\S
  6]{BKT2}). The element of $W^{(k)}$ corresponding to $\la$ is
denoted by $w_\la$, and lies in $W_n$ if and only if $\la$ fits inside
an $(n-k)\times(n+k)$ rectangle; the set of all such $k$-strict
partitions is denoted by $\cP(k,n)$.  Suppose that $w_\la =
(w_1,\ldots,w_n)\in W_n$, so that $0<w_1<\cdots < w_k$ and
$w_{k+1}<\cdots < w_n$. Then we deduce from \cite[\S 4.1 and
\S  4.4]{BKT1} that the bijection is given by the equations
\begin{equation}
\label{equtn}
\la_i=\begin{cases} 
k+|w_{k+i}| & \text{if $w_{k+i}<0$}, \\
\#\{p\leq k\, :\, w_p> w_{k+i}\} & \text{if $w_{k+i}>0$}.
\end{cases}
\end{equation}
Moreover, we have
\begin{equation}
\label{Cweq}
\cC(\la) = \{ (i,j)\in \Delta^\circ \ |\ w_{k+i}+ w_{k+j} < 0 \}
\end{equation}
and
\begin{equation}
\label{css}
\beta_j(\la)=\begin{cases}
w_{k+j}+1 & \text{if $w_{k+j}<0$}, \\
w_{k+j} & \text{if $w_{k+j}>0$}.
\end{cases}
\end{equation}

We are now ready to justify that the formula displayed in Theorem
\ref{cvthm} coincides with the equivariant Chevalley formula for the
products $\sigma_1\cdot \sigma_\la$ in the stable ring
$\IH_T(\IG_k)$. It follows from \cite[Thm.\ 1.1]{BKT1} and \cite[Lemma
  6.8]{IMN1} that
\[
\sigma_1\cdot \sigma_\la = (\sigma_1\vert_{w_\la})\cdot \sigma_\la 
+ \sum_{\la \to \mu} e_{\lambda\mu} \, \sigma_\mu,
\]
where $\sigma_1\vert_{w_\la}\in \Z[t]$ denotes the localization of the
stable equivariant Schubert class $\sigma_1$ at the torus-fixed point
indexed by $w_\la$. The polynomial $\sigma_1\vert_{w_\la}$ may be
computed e.g.\ as the image of the type C double Schubert polynomial
representing $\sigma_1$ under the homomorphism $\Phi_{w_\la}$ of
\cite[\S 6.1]{IMN1}; this gives
\begin{equation}
\label{inter}
\sigma_1\vert_{w_\la} = 2\sum_{j=1}^{\ell_k(\la)}t_{\la_j-k} + \sum_{j=1}^k(t_{w_j}-t_j).
\end{equation}
The fact that the term (\ref{inter}) coincides with the $t$-linear
coefficient in equation (\ref{cr}) amounts to the identity
\begin{equation}
\label{agree}
\sum_{j=1}^kt_{w_j} + \sum_{j=1}^{\ell_k}t_{\la_j-k} + 
\sum_{j=\ell_k+1}^{\ell} t_{\beta_j(\la)} = \sum_{j=1}^{k+\ell}t_j.
\end{equation}
The equations (\ref{css}) imply that 
\[
 \sum_{j=1}^{\ell_k}t_{\la_j-k} + 
\sum_{j=\ell_k+1}^{\ell} t_{\beta_j(\la)} =
\sum_{j=1}^{\ell_k}t_{|w_{k+j}|} +
\sum_{j=\ell_k+1}^{\ell} t_{w_{k+j}}. 
\]
Since $w$ is a signed {\em permutation}, the equality (\ref{agree}) follows.

\medskip

A theorem of Mihalcea \cite[Cor.\ 8.2]{Mi1} proves that the
equivariant Chevalley rule characterizes the $\Z[t]$-algebra
$\HH^*_{T_n}(\IG(n-k,2n))$, when it can be verified for a given
$\Z[t]$-basis of the latter ring. However, an analogous
characterization is not available for the stable equivariant
cohomology ring $\IH_T(\IG_k)$. To complete the proof that the
polynomials $\Ti_\la(c\, |\, t)$ represent equivariant Schubert
classes, Theorem \ref{cvthm} could be combined with a {\em vanishing
  theorem} that determines their behavior under the natural projection
$\pi_n$ of \S \ref{geomap} to the finite ring
$\HH^*_{T_n}(\IG(n-k,2n))$. The vanishing is easy to show using Lemma
\ref{ctlem} in the Lagrangian case when $k=0$, and similarly for
maximal orthogonal and type A Grassmannians. This method gives new,
intrinsic proofs that double Schur $S$-, $Q$-, and $P$-polynomials
represent equivariant Schubert classes, which apply to equivariant
{\em quantum} cohomology as well, along the lines of \cite{Mi2,
  IMN2}. It would be desirable to extend these arguments to arbitrary
isotropic Grassmannians; for the non-equivariant case, see \cite{BKT3,
  BKT4}.

\section{Examples and relations with other polynomials}
\label{firstexs}

In this section, we will prove some general facts about the 
double theta polynomials $\Ti_\la(c \, |\, t)$ and study how
they specialize and relate to earlier formulas.

\subsection{}
Let $\al=\{\al_j\}_{1\leq j\leq \ell}$ and $\rho=\{\rho_j\}_{1\leq
j\leq \ell}$ be two integer sequences. Consider the following Schur
type determinant:
\[
S_{\al}^\rho(c\, |\, t):= \prod_{i<j}(1-R_{ij})\, c_\al^\rho=
 \det \left(c^{\rho_i}_{\al_i+j-i}\right)_{1\leq i,j \leq \ell}
\]
which specializes when $t=0$ to the {\em Schur $S$-polynomial}
$$S_\al(c):= \det(c_{\al_i+j-i})_{1\leq i,j \leq \ell}.$$
Furthermore, define the Schur type Pfaffian:
\[
Q_{\al}^\rho(c\, |\, t):= \prod_{i<j}\frac{1-R_{ij}}{1+R_{ij}}
\, c_\al^\rho = 
\Pf\left(\frac{1-R_{12}}{1+R_{12}}\,
c^{\rho_i,\rho_j}_{\al_i,\al_j}\right)_{i<j}.
\]
(The matrix in the Pfaffian has size $r\times r$, where $r$ is the
least even integer such that $r\geq \ell$.) This specializes when
$t=0$ to the {\em Schur $Q$-polynomial}
\[
Q_\al(c) := \Pf\left(\frac{1-R_{12}}{1+R_{12}}
c_{\al_i,\al_j}\right)_{i<j},
\]
where the entries in the Pfaffian satisfy
\[
\frac{1-R_{12}}{1+R_{12}}\,
c_{\al_i,\al_j} = c_{\al_i}c_{\al_j}-2c_{\al_i+1}c_{\al_j-1}+2c_{\al_i+2}c_{\al_j-2}-\cdots
\]

Let $\la=(\la_1,\ldots,\la_\ell)$ be 
a $k$-strict partition of length $\ell$, and define 
\[
\cC_\ell(\la):=\{(i,j)\in \Delta^\circ \, |\,  j \leq \ell 
\ \ \text{and} \ \ \la_i+\la_j > 2k+j-i\}.
\]
and
\[
\cA_\ell(\la):=\{(i,j)\in \Delta^\circ \, |\, j \leq \ell 
\ \ \text{and} \ \ \la_i+\la_j \leq 2k+j-i\}.
\]
Recall that a {\em multiset} is a set in which members are allowed to
appear more than once, with finite multiplicity.  The next result
generalizes formulas found in \cite[Prop.\ 5.9]{BKT2} and \cite[Prop.\
2]{BKT3} to the case of double theta polynomials.

\begin{prop}
\label{kstrictS}
For any $k$-strict partition $\la$ of length $\ell$, we have
\[
\Theta_{\la}(c\, |\, t) = \prod_{(i,j)\in \cC_\ell(\la)}\,
(1-R_{ij}+R_{ij}^2-\cdots)\,S^{\be(\la)}_{\la}(c\, |\, t)
\]
and 
\[
\Theta_{\la}(c\, |\, t) = \prod_{(i,j)\in \cA_\ell(\la)}\,
(1+R_{ij})\,Q^{\be(\la)}_{\la}(c\, |\, t) = 
\sum_{\nu\in \cN} Q^{\be(\la)}_{\nu}(c\, |\, t)
\]
in $\Z[c,t]$, where $\cN$ is the multiset of integer vectors defined by
  \[
  \cN := \left\{\prod_{(i,j)\in S}R_{ij}\,\la \mid S\subset
    \cA_\ell(\la)\right\}.
  \]
\end{prop}
\begin{proof}
The equalities are an immediate consequence of the formal raising
operator identities
\[
R_\ell^\la = \prod_{(i,j)\in \cC_\ell(\la)}(1+R_{ij})^{-1}\prod_{1\leq
  i<j\leq\ell} (1-R_{ij})
\]
and 
\[
R_\ell^\la = \prod_{(i,j)\in\cA_\ell(\la)}(1+R_{ij}) \prod_{1\leq
  i<j\leq\ell}\frac{1-R_{ij}}{1+R_{ij}}.
\]
\end{proof}

\subsection{}
\label{extremecases}
We next examine the extreme cases when
the parts of the partition $\la$ are all small or all large, 
respectively, with respect to $k$.

If $\la_i+\la_j\leq 2k+j-i$ for all $i<j$, then we have 
\begin{equation}
\label{Schurdet}
\Theta_\la(c \, |\, t) = S^{\be(\la)}_\la(c\, |\, t) 
= \det \left(c^{k+i-\la_i}_{\la_i+j-i}\right)_{1\leq i,j \leq \ell}.
\end{equation}
If $\la_i\leq k$ for all $i\geq 1$, the determinant in
(\ref{Schurdet}) is a {\em double Schur $S$-polynomial}.  Suppose that
$x=(x_1,x_2,\ldots)$ is another list of commuting independent
variables.  The specialization of $\Theta_\la(c \, |\, t)$ which maps
each $c_r$ to $e_r(x_1,\ldots,x_k)$ is called a factorial Schur
function, and was studied in \cite{BL} and \cite[I.3, Ex.\ 20]{M}.

\begin{prop}
\label{Aprop}
Suppose that $\la_i+\la_j\leq 2k+j-i$ for all $i<j$. Then we have
\begin{equation}
\label{Afacteq}
\Theta_{\la}(c\, |\, t) = \sum_{\mu\subset\la}S_{\mu}(c)
\det \left(h^{k+i-\la_i}_{\la_i-\mu_j+j-i}(-t)\right)_{1\leq i,j \leq \ell}
\end{equation}
in $\Z[c,t]$, where the sum is over all partitions $\mu$ contained in $\la$.
\end{prop}
\begin{proof}
Set $\delta_\ell = (\ell-1,\ldots,1,0)$.  For any composition $\al$,
we have $S_{\la-\al}(c)=0$ in $\Z[c]$, unless $\la-\al +\delta_\ell =
w(\mu+\delta_\ell)$ for some partition $\mu\subset \la$, in which case
$S_{\la-\al}(c) = (-1)^wS_{\mu}(c)$. For any raising operator $R$, we
have
\[
R\,c^{\be(\la)}_{\la} = 
\sum_{\al\geq 0}c_{R\la-\al}\,h^{\be(\la)}_{\al}(-t) 
= \sum_{\al\geq 0}(R\,c_{\la-\al})\,h^{\be(\la)}_{\al}(-t)
\]
where the sums are over all compositions $\al$.
We therefore obtain
\begin{gather*}
\Theta_{\la}(c\, |\, t) =   \prod_{i<j}(1-R_{ij})\, c^{\be(\la)}_{\la} = 
\sum_{\al\geq 0} S_{\la-\al}(c)\, h^{\be(\la)}_{\al}(-t) \\
= \sum_{\mu\subset\la}S_\mu(c)\sum_{w\in S_\ell}(-1)^w 
h^{\be(\la)}_{\la+\delta_\ell - w(\mu + \delta_\ell)}(-t) 
= \sum_{\mu\subset\la}S_\mu(c)
\det \left(h^{k+i-\la_i}_{\la_i-\mu_j+j-i}(-t)\right)_{1\leq i,j \leq \ell}.
\end{gather*}
\end{proof}

Suppose now that $\la_i>k$ for all nonzero parts $\la_i$. Then we have
\begin{equation}
\label{SchurPf}
\Theta_\la(c\, |\, t) = Q^{\be(\la)}_\la(c\, |\, t) 
= \Pf\left(\frac{1-R_{12}}{1+R_{12}}\,
c^{k+1-\la_i,k+1-\la_j}_{(\la_i,\la_j)}\right)_{i<j}.
\end{equation}
In 2000, Kazarian \cite{Ka} proved a multi-Pfaffian formula for those
equivariant Schubert classes $[X_\la]^T$ on isotropic Grassmannians
where all the subbundles $F_j$ involved in the defining conditions
(\ref{Xlam}) are isotropic. It is easily seen that the Pfaffian
(\ref{SchurPf}) maps to Kazarian's formula under the geometrization
map $\pi_n$ of \S \ref{geomap}.

When $k=0$, the Pfaffian in
(\ref{SchurPf}) becomes the {\em double Schur $Q$-polynomial}
\begin{equation}
\label{dblq}
Q_\la(c\, |\, t) = \Pf\left(\frac{1-R_{12}}{1+R_{12}}\,
c^{1-\la_i,1-\la_j}_{(\la_i,\la_j)}\right)_{i<j}.
\end{equation}
Define the Schur $Q$-functions $q_j(x)$ by the generating series
\[
\prod_{i=1}^{\infty}\frac{1+x_iz}{1-x_iz} 
= \sum_{j=0}^{\infty}q_j(x)z^j.
\]
Then the specialization of $Q_\la(c \, |\, t)$ which maps $c_r$ to
$q_r(x)$ for every $r\geq 1$ may be identified with the factorial
Schur $Q$-functions introduced by Ivanov \cite{I}. The connection
between Ivanov's theory and the equivariant cohomology of the
Lagrangian Grassmanian $\LG(n,2n)$ was pointed out by 
Ikeda \cite{Ik}.

\begin{prop}
\label{Cprop}
Suppose that $\la_i>k$ for $1\leq i\leq \ell$. Then we have
\begin{equation}
\label{Cfacteq}
\Theta_{\la}(c\, |\, t) = \sum_{\mu\subset\la}Q_{\mu}(c)
\det \left(h^{k+1-\la_i}_{\la_i-\mu_j}(-t)\right)_{1\leq i,j \leq \ell}
\end{equation}
in $C^{(k)}[t]$, where the
sum is over all strict partitions $\mu\subset \la$ with $\mu_i>k$ for 
$1\leq i \leq \ell$.
\end{prop}
\begin{proof}
Arguing as in the proof of Proposition \ref{Aprop}, we obtain
\[
\Theta_{\la}(c\, |\, t) =  \sum_{\al\geq 0} Q_{\la-\al}(c)\, h^{\be(\la)}_{\al}(-t)
\]
where the sum is over all compositions $\alpha$. Observe that the
polynomial $h^{\be(\la)}_{\al}(-t)=e^{-\be(\la)}_{\al}(-t)$ vanishes
unless $\al_j\leq -\be_j(\la)=\la_j-k-1$ for each $j$, which implies
that $\la_j-\al_j > k$ for all $j$. Therefore, according to
\cite[Lemma 1.3]{BKT2}, the $Q$-polynomial $Q_{\la-\al}(c)$ is
alternating in the indices $\la_j-\al_j$. It follows that
\[
\Theta_{\la}(c\, |\, t) = \sum_{\mu\subset\la}Q_{\mu}(c)
\sum_{w\in S_\ell}(-1)^w 
e^{-\be(\la)}_{\la- w(\mu)}(-t) = 
\sum_{\mu\subset\la}Q_{\mu}(c)
\det \left(e^{\la_i-k-1}_{\la_i-\mu_j}(-t)\right)_{1\leq i,j \leq \ell}
\]
where the sums over $\mu\subset\la$ are as in the statement of the proposition.
\end{proof}

\subsection{} 
\label{compare}
The general degeneracy locus formulas of \cite[Eqn.\ (26) and
Thm.\ 3]{T2} produce many different answers for each indexing Weyl
group element $w$, depending on a choice of two compatible sequences
${\mathfrak a}$ and ${\mathfrak b}$,
which specify the symmetries of the resulting expression. However,
once this latter choice is fixed, these formulas are {\em uniquely
  determined} (cf.\ \cite[\S 5.3]{T3}). We recall here the formula
given in \cite[Ex.\ 12(b)]{T3} for the equivariant Schubert
classes on symplectic Grassmannians.

The symmetric group $S_n$ is the subgroup of $W_n$ generated by the
transpositions $s_i$ for $1\leq i \leq n-1$; we let $S_\infty :=
\cup_nS_n$ be the corresponding subgroup of $W_\infty$. For every
permutation $u\in S_\infty$, let $\AS_u(t)$ denote the type A Schubert
polynomial of Lascoux and Sch\"utzenberger \cite{LS} indexed by $u$
(our notation follows \cite[\S 5]{T3}). We say that a factorization
$w_\la=uv$ in $W_\infty$ is reduced if
$\ell(w_\la)=\ell(u)+\ell(v)$. In any such factorization, the right
factor $v=w_\mu$ is also $k$-Grassmannian for some $k$-strict
partition $\mu$.

Consider the polynomial
\begin{equation}
\label{fstrep}
\Omega_\la(c\, |\, t):=\sum_{uw_\mu=w_\la}\Ti_\mu(c)\AS_{u^{-1}}(-t)
\end{equation}
where the sum is over all reduced factorizations $uw_\mu=w_\la$ with
$u\in S_\infty$. One can show that the partitions $\mu$ which appear
in (\ref{fstrep}) are all contained in $\la$ and satisfy
$\ell_k(\mu)=\ell_k(\la)=\ell_k$. It is furthermore proved in
\cite{T2} that $\Omega_\la(c\, |\, t)$ represents the stable equivariant
Schubert class $\sigma_\la$ indexed by $\la$. Theorem \ref{mainthm}
implies that $\Omega_\la(c\, |\, t)$ must agree with $\Ti_\la(c\, |\,
t)$ up to the relations (\ref{relations}) among the variables $c_r$ in
$C^{(k)}$ (see Corollary \ref{comp}). We give a direct proof
of this below, in the two extreme cases considered in \S 
\ref{extremecases}.

If $\la_i\leq k$ for all $i$, then $w_\la$ is a Grassmannian element
of $S_\infty$, and hence is {\em fully commutative} in the sense of
\cite{St}. There is a one-to-one correspondence between reduced
factorizations $w_\la=uw_\mu$ and partitions $\mu\subset \la$. For
each such $\mu$, we have $\Theta_\mu(c)=S_\mu(c)$, and the Schubert
polynomial $\AS_{u^{-1}}(-t)$ is a flagged skew Schur polynomial. The
latter may be computed using \cite[Thm.\ 2.2]{BJS}, which proves that
$\AS_{u^{-1}}(-t)= \det
\left(h^{k+i-\la_i}_{\la_i-\mu_j+j-i}(-t)\right)_{1\leq i,j \leq
  \ell}$.  Formula (\ref{Afacteq}) shows that in this case,
$\Omega_\la(c\, |\, t)=\Ti_\la(c\, |\, t)$ as polynomials in
$\Z[c,t]$.

If $\la_i> k$ for $1\leq i\leq \ell$, then $w_\la$ will not be a fully
commutative element of $W_\infty$, in general. However there is a
one-to-one correspondence between reduced factorizations
$w_\la=uw_\mu$ with $u\in S_\infty$ and $k$-strict partitions
$\mu\subset \la$ with $\ell_k(\mu)=\ell_k(\la)=\ell$. For each such
$\mu$, we have $\Theta_\mu(c)=Q_\mu(c)$, and can use e.g. the reduced
words for $w_\la$ and $w_\mu$ given in \cite[Ex.\ 9]{T1} to deduce
that $u$ is fully commutative.  According to \cite[Thm.\ 2.2]{BJS}
again, the Schubert polynomial $\AS_{u^{-1}}(-t)$ is equal to
$\det \left(h^{k+1-\la_i}_{\la_i-\mu_j}(-t)\right)_{1\leq i,j \leq
  \ell}$.  We thus see that $\Omega_\la(c\, |\, t)$ agrees with the
right hand side of formula (\ref{Cfacteq}), and hence $\Omega_\la(c\,
|\, t)=\Ti_\la(c\, |\, t)$ in $C^{(k)}[t]$.

\begin{example}
Let $k=0$, so that the partition $\la$ is strict, and
$\Ti_\la(c)$ and $\Ti_\la(c\, |\, t)$ specialize to the single and
double $Q$-polynomials $Q_\la(c)$ and $Q_\la(c\, |\, t)$,
respectively. Following \cite[Ex.\ 12(b)]{T3}, we have
\begin{equation}
\label{Qeq}
\Omega_\la(c\, |\, t) = \sum_{\mu\subset\la}
Q_\mu(c)\det\left(e^{\la_i-1}_{\la_i - \mu_j}(-t)\right)_{1\leq i,j\leq \ell}
\end{equation}
summed over all strict partitions $\mu\subset\la$ with
$\ell(\mu)=\ell(\la)=\ell$. Taking
$\la=(3,1)$ in (\ref{Qeq}) gives
\[
\Omega_{3,1}(c\, |\, t) = Q_{3,1}(c) - Q_{2,1}(c)(t_1+t_2) = 
(c_3c_1-2c_4)-(c_2c_1-2c_3)(t_1+t_2).
\]
On the other hand, equation (\ref{dblq}) above gives
\begin{align*}
Q_{3,1}(c\, |\, t) &=  \frac{1-R_{12}}{1+R_{12}}\, 
c^{-2,0}_{(3,1)} = 
c_3^{-2}c_1^0 - 2 \, c_4^{-2}c_0^0 \\
&= (c_3-c_2(t_1+t_2)+c_1t_1t_2)c_1 - 
2\,(c_4-c_3(t_1+t_2) + c_2t_1t_2).
\end{align*}
Observe that $Q_{3,1}(c\, |\, t) - \Omega_{3,1}(c\, |\, t) =
(c_1^2-2c_2)t_1t_2$, and thus the two polynomials are not equal in
$\Z[c,t]$. However, since the relation $c_1^2=2c_2$ holds in
$C^{(0)}$, we have agreement in $C^{(0)}[t]$, as expected.
\end{example}

\section{Divided difference operators on $\Z[c,t]$}
\label{dds}

There is an action of $W_\infty$ on $\Z[c,t]$ by ring automorphisms,
defined as follows. The simple reflections $s_i$ for $i>0$ act
by interchanging $t_i$ and $t_{i+1}$ and leaving all the remaining
variables fixed. The reflection $s_0$ maps $t_1$ to $-t_1$, fixes the
$t_j$ for $j\geq 2$, and satisfies, for each $p\geq 0$,
\[
s_0(c_p) = c_p+2\sum_{j=1}^p(-t_1)^jc_{p-j}.
\]
The latter equation can be written in the form 
\begin{equation}
\label{concise}
s_0\left(\sum_{p=0}^{\infty}c_pu^p\right) = 
\frac{1-t_1u}{1+t_1u}\cdot\sum_{p=0}^{\infty}c_pu^p
\end{equation}
where $u$ is a formal variable with $s_i(u)=u$ for each $i$.
It follows that
\[
s_0s_1s_0s_1\left(\sum_{p=0}^{\infty}c_pu^p\right) = 
\frac{(1-t_1u)(1-t_2u)}{(1+t_1u)(1+t_2u)}\cdot
\sum_{p=0}^{\infty}c_pu^p =s_1s_0s_1s_0\left(\sum_{p=0}^{\infty}c_pu^p\right),
\]
and therefore that the braid relations for $W_\infty$ are satisfied on
$\Z[c,t]$.

For every $i\geq 0$, we define the {\em divided difference operator}
$\partial_i$ on $\Z[c,t]$ by
\[
\partial_0f := \frac{f-s_0f}{2t_1}, \qquad
\partial_if := \frac{f-s_if}{t_{i+1}-t_i}, \ \ \ \ \text{if $i\geq 1$}.
\]
The operators $\partial_i$ correspond to the {\em left} divided
differences $\delta_i$ studied in \cite{IMN1} and \cite{IM}. However,
the definition given here is more general than that found in
op.\ cit., since the $\partial_i$ act on the polynomial ring $\Z[c,t]$
rather than on its quotient $C^{(k)}[t]$. Note that for each $i\geq 0$, the 
$\partial_i$ satisfy the Leibnitz rule
\[
\partial_i(fg) = (\partial_if)g+(s_if)\partial_ig.
\]

We next recall from \cite[\S 5.1]{IM} some important identities 
satisfied by these operators, with statements and proofs in our setting.

\begin{lemma}
\label{lem1}
We have
\[
s_i(c_p^r) = 
\begin{cases}
c_p^r & \text{if $r \neq \pm i$}, \\
c_p^{i+1}+t_ic_{p-1}^{i+1} & \text{if $r= i > 0$}, \\
c_p^{-i+1}-t_{i+1}c_{p-1}^{-i+1} & \text{if $r= -i \leq 0$}.
\end{cases}
\]
\end{lemma}
\begin{proof}
Since $c_p^r$ is symmetric in $(t_1,\ldots,t_{|r|})$ when $r\neq 0$,
the identity $s_i(c_p^r) = c_p^r$ for $r\neq \pm i$ is clear if
$i>0$. For $r>0$, we have
\begin{equation}
\label{EE}
\sum^{\infty}_{p=0}c_p^ru^p = \left(\sum_{i=0}^{\infty}c_iu^i\right)\prod_{j=1}^r
\frac{1}{1+t_ju}
\end{equation}
and we apply $s_0$ to both sides and use (\ref{concise}) to prove that
$s_0(c_p^r)=c_p^r$ for all $p$; the argument when $r<0$ is similar.
The rest of the proof is straightforward.
\end{proof}

\begin{lemma}
\label{lem2}
Suppose that $p,r\in \Z$.

\medskip
\noin
{\em a)} For all $i\geq 0$, we have 
\[
\partial_ic_p^r= 
\begin{cases}
c_{p-1}^{r+1} & \text{if $r=\pm i$}, \\
0 & \text{otherwise}.
\end{cases}
\]
\medskip
\noin
{\em b)} For all $i\geq 1$, we have 
\[
\partial_i(c_p^{-i}c_q^i) = c_{p-1}^{-i+1}c_q^{i+1} +
c_p^{-i+1}c_{q-1}^{i+1}.
\]
\end{lemma}
\begin{proof}
For part (a), observe that if $r \neq \pm i$, then $\partial_ic_p^r=0$
by Lemma \ref{lem1}. If $r=i>0$, then we compute easily using (\ref{EE})
that
\[
\partial_i\left(\sum_{p=0}^\infty c_p^r u^p \right) = 
\left(\sum_{p=0}^\infty c_p^r u^{p+1} \right)\prod_{j=1}^{r+1}\frac{1}{1+t_ju}
\]
from which the desired result follows. Work similarly when $r=i=0$ or 
$r=-i<0$. 

For part (b), we use the Leibnitz rule and Lemmas \ref{ctlem} and \ref{lem1}
to compute
\begin{align*}
\partial_i(c_p^{-i}c_q^i) &= \partial_i(c_p^{-i})c_q^i + s_i(c_p^{-i}) 
\partial_i(c_q^i) \\
&= c_{p-1}^{-i+1}c_q^i + (c_p^{-i+1}-t_{i+1}c_{p-1}^{-i+1})c_{q-1}^{i+1} \\
&= c_{p-1}^{-i+1}(c_q^i - t_{i+1}c_{q-1}^{i+1}) + c_p^{-i+1}c_{q-1}^{i+1} \\
&= c_{p-1}^{-i+1}c_q^{i+1} + c_p^{-i+1}c_{q-1}^{i+1}.
\end{align*}
\end{proof}

The following result was proved differently in \cite{IM}, working in a
ring ${\mathcal R}^{(k)}_{\infty}$ which is isomorphic to
$C^{(k)}[t]$, and not in the polynomial ring $\Z[c,t]$.

\begin{prop}
\label{uniq}
Let $\la$ and $\mu$ be $k$-strict partitions such that 
$|\la|=|\mu|+1$ and $w_\la=s_iw_{\mu}$ for some simple
reflection $s_i\in W_\infty$. Then we have
\[
\partial_i\Ti_\la(c\,|\, t)  = \Ti_{\mu}(c\,|\, t)
\]
in $\Z[c,t]$.
\end{prop}
\begin{proof}
Let $w=w_\la$, $\be=\be(\la)$, and
$\be'=\be(\mu)$. Following \cite[Lemmas 3.4, 3.5]{IM}, there are 4
possible cases for $w$, discussed below. In each case, we have
$\mu\subset\la$, so that $\mu_p=\la_p-1$ for some $p\geq 1$ and
$\mu_j=\la_j$ for all $j\neq p$, and the properties listed are an easy
consequence of equations (\ref{equtn}), (\ref{Cweq}), and (\ref{css}).
 
\medskip
\noin (a) $w = (\cdots \ov{1} \cdots)$. In this case we have $i=0$,
$\cC(\la)=\cC(\mu)$, $\be_p=i$, $\be'_p= i + 1$, while $\be_j=\be'_j$
for all $j \neq p$.

\medskip
\noin
(b) $w = (\cdots i+1 \cdots i \cdots)$.  In this case 
$\cC(\la)=\cC(\mu)$, $\be_p=i$, $\be'_p=i+1$, 
and $\be_j=\be'_j$ for all $j \neq p$.

\medskip
\noin (c) $w = (\cdots i \cdots \ov{i+1} \cdots)$. In this case
$\cC(\la)=\cC(\mu)$, $\be_p = -i$ and $\be'_p=-i+1$,
and $\be_j=\be'_j$ for all $j \neq p$.

\medskip
\noin (d) $w = (\cdots \ov{i+1} \cdots i \cdots)$. In this case
$\cC(\la)=\cC(\mu)\cup\{(p,q)\}$, where $w_{k+p}=-i-1$ and
$w_{k+q}=i$. It follows that $\be_p=-i$ and $\be_q=i$. We see
similarly that $\be'_p=-i+1$ and $\be'_q=i+1$, while $\be_j=\be'_j$
for all $j \notin \{p,q\}$.

\medskip
In cases (a), (b), or (c), it follows using the Leibnitz rule and
Proposition \ref{lem2}(a) that for any integer sequence
$\al=(\al_1,\ldots,\al_\ell)$, we have
\begin{align*}
\partial_i c^{\be(\la)}_\al &=
c^{(\be_1,\ldots,\be_{p-1})}_{(\al_1,\ldots,\al_{p-1})}
\left(\partial_i(c^{\be_p}_{\al_p})c^{(\be_{p+1},\ldots,\be_\ell)}_{(\al_{p+1},\ldots,\al_\ell)}
+s_i(c^{\be_p}_{\al_p})\partial_i
(c^{(\be_{p+1},\ldots,\be_\ell)}_{(\al_{p+1},\ldots,\al_\ell)})\right)
\\ &=c^{(\be_1,\ldots,\be_{p-1})}_{(\al_1,\ldots,\al_{p-1})}
\left(c^{\be_p+1}_{\al_p-1}c^{(\be_{p+1},\ldots,\be_\ell)}_{(\al_{p+1},\ldots,\al_\ell)}
+s_i(c^{\be_p}_{\al_p})\cdot 0\right)
=c^{(\be_1,\ldots,\be_p+1,\ldots,\be_\ell)}_{(\al_1,\ldots,\al_p-1,\ldots,\al_\ell)}
= c^{\be(\mu)}_{\al-\epsilon_p}.
\end{align*}
Since $\la-\epsilon_p=\mu$,
we deduce that if $R$ is any raising operator, then
\[
\partial_i R \,c^{\be(\la)}_\la = \partial_i c^{\be(\la)}_{R\la} = 
c^{\be(\mu)}_{R\la-\epsilon_p} = R \,c^{\be(\mu)}_\mu.
\]
As $R^\la= R^\mu$, we conclude that 
\[
\partial_i \Ti_\la(c\,|\, t)  = \partial_i R^\la c^{\be(\la)}_\la = 
R^\mu c^{\be(\mu)}_\mu = \Ti_\mu(c\,|\, t).
\]

In case (d), it follows from the Leibnitz rule as in the proof of 
Proposition \ref{lem2}(b) that for any integer sequence
$\al=(\al_1,\ldots,\al_\ell)$, we have
\begin{align*}
\partial_i c^{\be(\la)}_\al &= \partial_i
c^{(\be_1,\ldots,-i,\ldots,i,\ldots,\be_{\ell})}_{(\al_1,\ldots,\al_p,\ldots,\al_q,
  \ldots,\al_\ell)} \\ &=
c^{(\be_1,\ldots,-i+1,\ldots,i+1,\ldots,\be_{\ell})}_{(\al_1,\ldots,\al_p-1,\ldots,\al_q,\ldots,\al_\ell)}
+
c^{(\be_1,\ldots,-i+1,\ldots,i+1,\ldots,\be_{\ell})}_{(\al_1,\ldots,\al_p,\ldots,\al_q-1,\ldots,\al_\ell)}
= c^{\be(\mu)}_{\al-\epsilon_p} +
c^{\be(\mu)}_{\al-\epsilon_q}.
\end{align*}
Since $\la-\epsilon_p=\mu$, we deduce that if $R$ is any raising
operator, then
\[
\partial_i R \,c^{\be(\la)}_\la = \partial_i c^{\be(\la)}_{R\la} = 
c^{\be(\mu)}_{R\la-\epsilon_p} + c^{\be(\mu)}_{R\la-\epsilon_q} = 
R \,c^{\be(\mu)}_\mu + RR_{pq}\,c^{\be(\mu)}_\mu.
\]
As $R^\la+R^\la R_{pq} = R^\mu$, we conclude that
\[
\partial_i \Ti_\la(c\,|\, t)  = 
\partial_i R^\la c^{\be(\la)}_\la = 
R^\la c^{\be(\mu)}_\mu +R^\la R_{pq}c^{\be(\mu)}_\mu = 
R^\mu c^{\be(\mu)}_\mu = \Ti_\mu(c\,|\, t).
\]
\end{proof}

\section{The proof of Theorem \ref{mainthm}}
\label{gdds}

\subsection{The geometrization map}
\label{geomap}
Let  
\[
0 \to E'\to E \to E''\to 0
\]
be the tautological sequence (\ref{ses}) of vector bundles over
$\IG(n-k,2n)$, and define the subbundles $F_j$ of $E$ for $0\leq j
\leq 2n$ as in the introduction. Let $IM_n:=ET_n\times^{T_n}\IG$
denote the Borel mixing space for the action of the torus $T_n$ on
$\IG$. The $T_n$-equivariant vector bundles $E',E,E'',F_j$ over $\IG$
induce vector bundles over $IM_n$. Their Chern classes in
$\HH^*(IM_n,\Z)=\HH^*_{T_n}(\IG(n-k,2n))$ are called {\em equivariant 
Chern classes} and denoted by $c_p^T(E')$, $c^T_p(E)$, etc.

Recall that $c^T_p(E-E'-F_j)$ for $p\geq 0$ is defined by the total
Chern class equation
$$c^T(E-E'-F_j):=c^T(E)c^T(E')^{-1}c^T(F_j)^{-1}.$$ Let
$\tm_i:=-c^T_1(F_{n+1-i}/F_{n-i})$ for $1\leq i \leq n$.  Following
\cite[\S 10]{IMN1} and \cite[Thm.\ 3]{T2}, we define the {\em
  geometrization map} $\pi_n$ as the $\Z[t]$-algebra homomorphism
\[
\pi_n : C^{(k)}[t] \to \HH^*_{T_n}(\IG(n-k,2n))
\]
determined by setting
\begin{gather*}
\pi_n(c_p):= c^T_p(E-E'-F_n) \ \ \text{for all} \ p, \\ 
\pi_n(t_i):= \begin{cases}
\tm_i & \text{if $1\leq i\leq n$}, \\
0 & \text{if $i>n$}. \end{cases}
\end{gather*}
Since $\tm_1\ldots,\tm_r$ are the (equivariant)
Chern roots of $F_{n+r}/F_n$ for $1\leq r \leq n$, it follows that
\begin{equation}
\label{geomeq}
\pi_n(c^r_p) = \sum_{j=0}^pc^T_{p-j}(E-E'-F_n)h^r_j(-\tm)
= c^T_p(E-E'-F_{n+r})
\end{equation}
for $-n \leq r \leq n$. This can be extended to $r\in \Z$ if we set
$F_j=F_{2n}=E$ for $j>2n$ and $F_j=0$ for $j<0$. For example, if $\la$
is any partition in $\cP(k,n)$, then we have
\begin{equation}
\label{Kaz}
\pi_n(\Ti_\la(c\, |\, t)) = \pi_n(R^\la\, c^{\be(\la)}_\la) =
R^\la\, c^T_\la(E-E'-F_{n+\be(\la)}),
\end{equation}
in agreement with formula (\ref{genKaz}).

The embedding of $W_n$ into $W_{n+1}$ defined in \S \ref{trans}
induces maps of equivariant cohomology rings
$\HH^*_{T_{n+1}}(\IG(n+1-k,2n+2)) \to \HH^*_{T_n}(\IG(n-k,2n))$ which
are compatible with the morphisms $\pi_n$. We thus obtain an induced
$\Z[t]$-algebra homomorphism
\[
\pi:C^{(k)}[t]\to \IH_T(\IG_k),
\]
which we will show has the properties listed in Theorem \ref{mainthm}.

\subsection{The equivariant Schubert class of a point}
\label{pointclass}

Fix a rank $n$ and let $$\la_0:=(n+k,n+k-1,\ldots,2k+1)$$ be the
$k$-strict partition associated to the $k$-Grassmannian element of
maximal length in $W_n$. In this short section we discuss three different
proofs that 
\begin{equation}
\label{pteq}
\pi_n(\Ti_{\la_0}(c\, |\, t)) = [X_{\la_0}]^{T_n}.
\end{equation}

Equation (\ref{SchurPf}) gives
\[
\Ti_{\la_0}(c\, |\, t) = Q_{\la_0}^{(1-n,2-n,\ldots,-k)}(c\, |\, t)
\]
and by applying (\ref{geomeq}) we obtain that
\begin{equation}
\label{Kresult}
\pi_n(\Ti_{\la_0}(c\, |\,t)) = Q_{\la_0}(E-E'-F_{(1,2,\ldots,n-k)}).
\end{equation}
Observe now that the right hand side of (\ref{Kresult}) coincides with
Kazarian's multi-Pfaffian formula \cite[Thm.\ 1.1]{Ka} for the
cohomology class of the degeneracy locus which correponds to
$[X_{\la_0}]^{T_n}$ (compare with Corollary \ref{maincor}). Therefore
(\ref{pteq}) follows directly from a known result, which is proved
geometrically in op.\ cit.

A second proof of (\ref{pteq}) is obtained by using the equality 
\begin{equation}
\label{2eq}
\Ti_{\la_0}(c\, |\, t) =
\sum_{uw_\mu=w_{\la_0}}Q_\mu(c)\AS_{u^{-1}}(-t) =\Omega_{\la_0}(c\,
|\, t)
\end{equation}
in $C^{(k)}[t]$, which is shown in \S \ref{compare}. We then appeal to
\cite[Eqn.\ (26) and Thm.\ 3]{T2}, which, in this special situation,
give the equality $\pi_n(\Omega_{\la_0}(c\, |\,
t))=[X_{\la_0}]^{T_n}$. Note that the proof of this in op.\ cit.\ uses the
fact that the single $Q$-polynomials $Q_\mu(c)$ in (\ref{2eq})
represent the corresponding cohomological Schubert classes on
$\IG$. This last result is a special case of the main theorem of
\cite{BKT2}.

Finally, a third proof of (\ref{pteq}) is provided by Ikeda and
Matsumura in \cite[\S 8.2]{IM}, starting from the Pfaffian formula
\cite[Thm.\ 1.2]{IMN1} for the equivariant Schubert class of a point
on the complete symplectic flag variety $\Sp_{2n}/B$. From this, using
the left divided differences, they derive a Pfaffian formula for the
top equivariant Schubert class on any symplectic partial flag variety.
In particular, one recovers the Pfaffian $\Ti_{\la_0}(c\, |\,t)$ which
represents the class $[X_{\la_0}]^{T_n}$.

\subsection{Proof of Theorem \ref{mainthm}} We have shown in
Proposition \ref{basisthm} that the $\Ti_\la(c\, |\, t)$ for $\la$
$k$-strict form a $\Z[t]$-basis of $C^{(k)}[t]$. Following \cite[\S
3.4]{IM}, for any $k$-strict partition $\la\in \cP(k,n)$, write
$w_{\la}w_{\la_0}=s_{a_1}\cdots s_{a_r}$ as a product of simple
reflections $s_{a_j}$ in $W_n$, with $r=|\la_0|-|\la|$. Since
$w_{\la_0}^2=1$, we deduce from Proposition \ref{uniq} that
\begin{equation}
\label{iteratepar}
\Ti_\la(c\,|\, t) = \partial_{a_1} \circ \cdots \circ 
\partial_{a_r}(\Ti_{\la_0}(c\, |\, t))
\end{equation}
holds in $\Z[c,t]$. 

The action of the operators $\partial_i$ on $\Z[c,t]$ induces an
action on the quotient ring $C^{(k)}[t]$. Moreover, the corresponding
left divided differences $\delta_i$ on $H_n:=\HH_{T_n}^*(\IG(n-k,2n))$
from \cite[\S 2.5]{IMN1} are compatible with the geometrization map
$\pi_n:C^{(k)}[t]\to H_n$. According to \cite[Prop.\ 2.3]{IMN1} (see
also \cite[Eqn.\ (62)]{T3}), we have $\delta_i([X_\la]^{T_n})=
[X_{\mu}]^{T_n}$ whenever $|\la|=|\mu|+1$ and $w_\la=s_iw_{\mu}$ for some
simple reflection $s_i$. It follows from this and equations
(\ref{pteq}) and (\ref{iteratepar}) that
\begin{equation}
\label{pineq}
\pi_n(\Ti_\la(c\, |\, t)) = [X_\la]^{T_n}. 
\end{equation}
The fact that $\pi_n(\Ti_\la(c\, |\, t))=0$ whenever
$\la\notin\cP(k,n)$, or equivalently $w_\la\notin W_n$, is now a
consequence of the vanishing property for equivariant Schubert classes
(see e.g.\ \cite[Prop.\ 7.7]{IMN1}).  The induced map
$\pi:C^{(k)}[t]\to \IH_T(\IG_k)$ satisfies $\pi(\Ti_\la(c\, |\,
t))=\sigma_\la$ for all $k$-strict $\la$, and is a $\Z[t]$-algebra
isomorphism because the $\Ti_\la(c\, |\, t)$ and $\sigma_\la$ for $\la$
$k$-strict form $\Z[t]$-bases of the respective algebras.

We are left with proving the last assertion in Theorem \ref{mainthm},
about the presentation of the $\Z[t]$-algebra $H_n$. Let $I_n$ be the
ideal of $C^{(k)}[t]$ generated by the $\Ti_\la(c\, |\, t)$ for
$\la\notin\cP(k,n)$, and $J_n$ be the ideal of $C^{(k)}[t]$ generated
by the relations (\ref{rel1}) and (\ref{rel2}). Since $J_n\subset
I_n$, there is a surjection of $\Z[t]$-algebras $$\psi: C^{(k)}[t]/J_n
\to C^{(k)}[t]/I_n\,.$$ We have established that $H_n\cong
C^{(k)}[t]/I_n$, so it suffices to show that $\Ker(\psi)=0$.

Consider the ideal $I$ of $\Z[t]$ generated by the $t_i$ for all
$i\geq 1$, and let $M:=C^{(k)}[t]/J_n$. When $k\geq 1$, we have
the formal identity
\begin{equation}
\label{ident}
\Ti_{(1^p)}(c) = \det(c_{1+j-i})_{1\leq i,j \leq p} 
\end{equation}
which is a specialization of (\ref{Schurdet}).  We deduce from
(\ref{ident}) and the presentation of $\HH^*(\IG(n-k,2n))$ given in
\cite[Thm.\ 1.2]{BKT1} that the polynomials $\Ti_\la(c)$ for
$\la\in\cP(k,n)$ generate $M/IM$ as a $\Z$-module. It follows from
\cite[Lemma 4.1]{Mi2} that the $\Ti_\la(c\, |\, t)$ for
$\la\in\cP(k,n)$ generate $M$ as a $\Z[t]$-module.  Now suppose that
$\psi$ maps a general element $\sum_\la a_\la(t)\Ti_\la(c\, |\, t)
+J_n$ of $M$ to zero, where $a_\la(t)\in \Z[t]$ and the sum is over
$\la\in \cP(k,n)$. Then $\sum_\la a_\la(t)\Ti_\la(c\, |\, t)\in
I_n$. Since $I_n$ is equal to the $\Z[t]$-submodule of $C^{(k)}[t]$
with basis $\Ti_\la(c\, |\, t)$ for $\la\notin \cP(k,n)$, we conclude
that $a_\la(t)=0$, for all $\la$. This completes the proof of Theorem
\ref{mainthm}, and Corollary \ref{maincor} follows from 
equations (\ref{Kaz}) and (\ref{pineq}).

\medskip
According to \cite{T2}, the polynomials $\Omega_{\la}(c\, |\, t)$ of
\S \ref{compare} represent the stable equivariant Schubert classes
$\sigma_\la$ in $\IH_T(\IG_k)$ under the geometrization map
$\pi$. The next result is therefore an immediate consequence
of Theorem \ref{mainthm}.

\begin{cor}
\label{comp}
Let $\la$ be any $k$-strict partition. Then we have
\begin{equation}
\label{TO}
\Ti_\la(c\, |\, t)=\sum_{uw_\mu=w_\la}\Ti_\mu(c)\AS_{u^{-1}}(-t)
\end{equation}
in the ring $C^{(k)}[t]$, where the sum is over all reduced factorizations 
$uw_\mu=w_\la$ with $u\in S_\infty$.
\end{cor}

We remark that the right hand side of (\ref{TO}) is the unique
expansion of $\Ti_\la(c\, |\, t)$ as a $\Z$-linear combination in 
the product basis $\{\Ti_\mu(c)\AS_u(-t)\}$ of $C^{(k)}[t]$, where 
$\mu$ ranges over all $k$-strict partitions and $u$ lies in $S_\infty$.
It is instructive to write equation (\ref{TO}) in the following way:
\begin{equation}
\label{endeq}
R^{\la}\,c^{\be(\la)}_{\la}=\sum_{uw_\mu=w_\la}(R^{\mu}\,c_{\mu})\AS_{u^{-1}}(-t).
\end{equation}


\begin{thebibliography}{HIMN}


\bibitem[AF]{AF} D. Anderson and W. Fulton : 
{\em Chern class formulas for classical-type degeneracy loci},
arXiv:1504.03615.


\bibitem[BL]{BL} L. C. Biedenharn and J. D. Louck :
{\em A new class of symmetric polynomials defined in terms of 
tableaux}, Adv. in Appl. Math. {\bf 10} (1989), 396--438.


\bibitem[BH]{BH} S. Billey and M. Haiman :
{\em Schubert polynomials for the classical groups},
J. Amer. Math. Soc. {\bf 8} (1995), 443--482.



\bibitem[BJS]{BJS} S. Billey, W. Jockusch and R. P. Stanley :
{\em Some combinatorial properties of Schubert polynomials},
J. Algebraic Combin. {\bf 2} (1993), 345--374.



\bibitem[BKT1]{BKT1} A. S. Buch, A. Kresch, and H. Tamvakis :
{\em Quantum Pieri rules for isotropic Grassmannians},
Invent. Math. {\bf 178} (2009), 345--405.


\bibitem[BKT2]{BKT2} A. S. Buch, A.  Kresch, and H. Tamvakis :
{\em A Giambelli formula for isotropic Grassmannians},
Selecta Math. (N.S.), to appear.


\bibitem[BKT3]{BKT3} A. S. Buch, A.  Kresch, and H. Tamvakis :
{\em Quantum Giambelli formulas for isotropic Grassmannians},
Math.\ Ann.\ {\bf 354} (2012), 801--812.

\bibitem[BKT4]{BKT4} A. S. Buch, A.  Kresch, and H. Tamvakis :
{\em A Giambelli formula for even orthogonal Grassmannians},
J. reine angew. Math. {\bf 708} (2015), 17--48.

\bibitem[F]{F} W. Fulton :
{\em Determinantal formulas for orthogonal and symplectic degeneracy
loci}, J. Differential Geom. {\bf 43} (1996), 276--290.

\bibitem[G]{G} G. Z. Giambelli :
{\em Risoluzione del problema degli spazi secanti}, Mem. R. Accad. Sci.
Torino (2) {\bf 52} (1902), 171--211.


\bibitem[Gr]{Gr} W. Graham :
{\em The class of the diagonal in flag bundles},  J. Differential Geom. 
{\bf 45} (1997), 471--487.

\bibitem[HIMN]{HIMN} T. Hudson, T. Ikeda, T. Matsumura, H. Naruse :
{\em Determinantal and Pfaffian formulas of $K$-theoretic Schubert
calculus}, arXiv:1504.02828.

\bibitem[I]{Ik} T. Ikeda : 
{\em Schubert classes in the equivariant cohomology of the Lagrangian 
Grassmannian}, Adv. Math. {\bf 215} (2007), 1--23. 

\bibitem[IM]{IM} T. Ikeda and T. Matsumura :
{\em Pfaffian sum formula for the symplectic Grassmannian}, 
Math. Z. {\bf 280} (2015), 269--306.


\bibitem[IMN1]{IMN1} T. Ikeda, L. C. Mihalcea, and H. Naruse :
{\em Double Schubert polynomials for the classical groups}, 
Adv. Math. {\bf 226} (2011), 840--886.

\bibitem[IMN2]{IMN2} T. Ikeda, L. C. Mihalcea, and H. Naruse :
{\em Factorial $P$- and $Q$-Schur functions represent equivariant
quantum Schubert classes}, arXiv:1402.0892.


\bibitem[IN]{IN} T. Ikeda and H. Naruse :
{\em Excited Young diagrams and equivariant Schubert calculus},
Trans. Amer. Math. Soc. {\bf 361} (2009), 5193--5221. 


\bibitem[Iv]{I} V. N. Ivanov : 
{\em Interpolation analogues of Schur $Q$-functions}, Zap. Nauchn. 
Sem. S.-Peterburg. Otdel. Mat. Inst. Steklov. (POMI) {\bf 307} (2004), 
Teor. Predst. Din. Sist. Komb. i Algoritm. Metody. 10, 99--119, 
281--282;  translation in J. Math. Sci. (N. Y.) {\bf 131} (2005),
5495--5507 .


\bibitem[Ka]{Ka} M. Kazarian :
{\em On Lagrange and symmetric degeneracy loci}, preprint, 
Arnold Seminar (2000); available at 
http://www.newton.ac.uk/preprints/NI00028.pdf.


\bibitem[KL]{KL} G. Kempf and D. Laksov :
{\em The determinantal formula of Schubert calculus},
Acta Math. {\bf 132} (1974), 153--162.



\bibitem[LS]{LS} A. Lascoux and M.-P. Sch\"{u}tzenberger :
{\em Polyn\^{o}mes de Schubert}, C. R. Acad. Sci. Paris S\'er. I
Math. {\bf 294} (1982), 447--450.


\bibitem[M]{M} I. G. Macdonald :
{\em Symmetric functions and Hall polynomials}, Second edition,
The Clarendon Press, Oxford University Press, New York, 1995.


\bibitem[Mi1]{Mi1} L. C. Mihalcea : 
{\em On equivariant quantum cohomology of homogeneous spaces:
Chevalley formulae and algorithms}, Duke Math. J. {\bf 140}
(2007), 321--350.


\bibitem[Mi2]{Mi2} L. C. Mihalcea : 
{\em Giambelli formulae for the equivariant quantum cohomology of the 
Grassmannian}, Trans. Amer. Math. Soc. {\bf 360} (2008), 2285--2301.


\bibitem[P]{P} P. Pragacz : 
{\em Algebro-geometric applications of Schur $S$- and $Q$-polynomials},
S\'{e}minare d'Alg\`{e}bre Dubreil-Malliavin 1989-1990, Lecture
Notes in Math. {\bf 1478} (1991), 130--191, Springer-Verlag, Berlin,



\bibitem[St]{St} J. R. Stembridge :
{\em Some combinatorial aspects of reduced words in finite Coxeter 
groups}, Trans. Amer. Math. Soc. {\bf 349} (1997), 1285--1332.



\bibitem[T1]{T1} H. Tamvakis : 
{\em Giambelli, Pieri, and tableau formulas via raising operators}, 
J. reine angew. Math. {\bf 652} (2011), 207--244.


\bibitem[T2]{T2} H. Tamvakis :
{\em A Giambelli formula for classical $G/P$ spaces}, 
J. Algebraic Geom. {\bf 23} (2014), 245--278.


\bibitem[T3]{T3} H. Tamvakis : 
{\em Giambelli and degeneracy locus formulas for classical $G/P$
spaces}, Mosc. Math. J., to appear.


\bibitem[T4]{T4} H. Tamvakis :
{\em Double eta polynomials and equivariant Giambelli formulas},
arXiv:1506.04441.


\bibitem[W]{W} E. Wilson : 
{\em Equivariant Giambelli formulae for Grassmannians}, Ph.D.\ thesis,
University of Maryland, 2010.


\bibitem[Y]{Y} A. Young : 
{\em On quantitative substitutional analysis VI}, Proc. Lond. 
Math. Soc. (2) {\bf 34} (1932), 196--230.


\end{thebibliography}
\end{document}